\documentclass[a4paper,reqno,12pt]{amsart}

\usepackage[
  top=3.5cm,
  bottom=3.5cm,
  left=2cm,
  right=2cm
]{geometry}

\usepackage[T1]{fontenc}
\usepackage[utf8]{inputenc}
\usepackage[english]{babel}
\usepackage{datetime}
\DeclareUnicodeCharacter{0304}{\={}}

\usepackage{amsmath,amsthm,amssymb,amsfonts,mathtools}
\usepackage[textwidth=2cm]{todonotes}
\usepackage{float}
\usepackage{enumitem}
\setlist[enumerate]{leftmargin=*}
\setlist[itemize]{leftmargin=*}
\usepackage{wrapfig}
\usepackage{comment}
\usepackage{xcolor}
\usepackage[numbers]{natbib}

\usepackage{hyperref}
\hypersetup{
  colorlinks = true,
  urlcolor   = magenta,
  linkcolor  = blue,
  citecolor  = red
}

\numberwithin{equation}{section}

\newtheorem{theorem}{Theorem}[section]
\newtheorem{lemma}[theorem]{Lemma}
\newtheorem{corollary}[theorem]{Corollary}
\newtheorem{proposition}[theorem]{Proposition}

\theoremstyle{definition}

\newtheorem{remark}[theorem]{Remark}

\newcommand{\RR}{\mathbb{R}}
\newcommand{\NN}{\mathbb{N}}

\newcommand{\ee}{\mathrm{e}}
\newcommand{\dd}{\mathrm{d}}
\newcommand{\SF}{\mathcal{F}_\sigma}
\newcommand{\Scal}{\mathcal{S}}
\newcommand{\BB}{\mathcal{B}}

\newcommand{\Id}{\mathrm{Id}}
\newcommand{\defeq}{:=}
\newcommand{\eqdef}{=:}
\newcommand{\tr}{\operatorname{tr}}

\newcommand{\PR}{\Psi_R}

\makeatletter
\def\@setaddresses{}
\makeatother

\excludecomment{liucomment}

\begin{document}
\title[Quantitative Fourier Restriction Estimates for Weyl Operators]
{Quantitative Fourier Restriction Estimates for Weyl Operators: Fourier-Support Dependence and Lower Bounds}
\author{Jie Liu}



	\begin{abstract}
The Weyl calculus associates a function \(a\) on phase space \(\RR^{2d}\) with the corresponding Weyl operator \(L_a\) acting on \(L^2(\RR^d)\). At \(p=2\), this correspondence is governed by an exact Hilbert--Schmidt identity. For \(p\neq2\), two-sided \(L^p\)--Schatten estimates are known for Paley--Wiener type symbols, with constants depending on the Fourier-support scale. We study this quantitative dependence, improve the known upper bounds, and show that in large ranges of \(p\) no support-independent global comparison can hold.

Let \(\SF\) denote the symplectic Fourier transform, and let \(u\in\mathcal E'(\RR^{2d})\) satisfy \(\operatorname{supp}u\subset\overline{B(z_0,R)}\), where \(R\geq1\). Then, for every \(1\leq p\leq\infty\) and every \(\varepsilon>0\), we prove
\[
\|L_{\SF u}\|_{\Scal_p}\lesssim_{d,p,\varepsilon}R^{(2d+1+\varepsilon)|1-2/p|}\|\SF u\|_{L^p(\RR^{2d})},
\]
together with the reverse estimate with the same power of \(R\). This sharpens the earlier exponential dependence \(\ee^{cR^2}\) obtained by Luef and Samuelsen and improves M\"uller's polynomial loss \(R^{(5d+2)|1-2/p|}\).  The main ingredient is a radial trace-class estimate based on the Hermite--Laguerre correspondence \(\rho(\varphi_k)=P_k\), which reduces the relevant Weyl operators to finite-rank Hermite projections. 

We also show that dependence on $R$ is unavoidable. Compactly supported examples obtained by truncating Laguerre functions yield polynomial lower bounds for the best comparison constants. These examples refine M\"uller's operator-norm example and give nontrivial lower bounds for a larger range of Schatten exponents, which can cover the full range \(1\le p\le\infty\) except for the Hilbert--Schmidt point \(p=2\) as  \(d\to\infty\). Consequently, in these ranges the corresponding \(L^p\)--Schatten estimate cannot hold globally with a constant independent of the Fourier support.

Moreover, for every fixed \(p>2\), the exponent in the reverse comparison estimate is asymptotically optimal as \(d\to\infty\).
\end{abstract}

    \maketitle
    
 \section{Introduction}\label{Section-intro}

The Weyl calculus provides a correspondence between functions on phase space and operators on $L^2(\RR^d)$; see, for instance, \cite{F89,G01,W84,T93}. At the Hilbert--Schmidt endpoint this correspondence is governed by the exact identity
\[
    \|L_a\|_{\Scal_2}=(2\pi)^{-d/2}\|a\|_{L^2(\RR^{2d})};
\]
see, for example, \cite[Theorem~1.2.1]{T93}.
For \(p\neq2\), this exact identity does not extend to arbitrary symbols. Previous results \cite{MV24,LS25,M26} show, however, that two-sided \(L^p\)--Schatten estimates are available under a compact Fourier-support assumption, with constants depending on the Fourier-support scale. The main purpose of this paper is to improve the known quantitative dependence on this scale and to construct explicit examples showing that such dependence cannot, in general, be eliminated.

We formulate the problem using the symplectic Fourier transform. For $z=(x,\xi)$ and $w=(y,\eta)$ in $\RR^{2d}$, let
\[
    \sigma(w,z)=\eta\cdot x-y\cdot\xi
\]
be the \emph{standard symplectic form}, and define
\[
    \SF F(w)=(2\pi)^{-d}\int_{\RR^{2d}}F(z)\ee^{-i\sigma(w,z)}\,\dd z.
\]
We also denote by $\rho$ the \emph{Schr\"odinger representation} of phase space,
\[
    [\rho(x,\xi)g](t)=\ee^{-\frac{i}{2}x\cdot\xi}\ee^{it\cdot\xi}g(t-x),
    \,\,\, g\in L^2(\RR^d),
\]
and, for a suitable function $F$ on $\RR^{2d}$, set
\[
    \rho(F)=(2\pi)^{-d}\int_{\RR^{2d}}F(z)\rho(z)\,\dd z.
\]
This construction extends in the distributional sense to $u\in\Scal'(\RR^{2d})$. With our normalization, it is related to \emph{Weyl quantization} by
\[
    \rho(u)=L_{\SF u};
\]
see, for instance, \cite{F89,G01,W84}. Consequently, if $u\in\mathcal E'(\RR^{2d})$, then $\SF u$ is a Paley--Wiener type symbol and $L_{\SF u}$ is a \emph{Weyl operator} whose symbol has compact Fourier support. The radius of $\operatorname{supp}u$ is precisely the Fourier-support scale that enters our estimates.

The problem is closely connected with Fourier restriction theory, which is one of the central themes in modern harmonic analysis. In its classical form, given a non-negative Radon measure $\mu$ on $\RR^n$, one asks for which exponents $p,q$ the estimate
\[
    \|\widehat f\|_{L^q(\mu)}\lesssim \|f\|_{L^p(\RR^n)}
\]
holds for all $f\in\Scal(\RR^n)$. By duality, this is equivalent to an extension estimate. If $G\in L^{q'}(\mu)$, define
\begin{equation}\label{Def-classicalextension}
    \mathcal E_\mu G(x)
    =\int_{\RR^n}\ee^{ix\cdot\xi}G(\xi)\,\dd\mu(\xi),
    \,\,\,x\in\RR^n.
\end{equation}
Then the restriction estimate is equivalent to
\begin{equation}\label{Est-classicalextension}
    \|\mathcal E_\mu G\|_{L^{p'}(\RR^n)}
    \lesssim \|G\|_{L^{q'}(\mu)}.
\end{equation}
Thus restriction estimates may also be viewed as bounds for oscillatory integral operators. 

Classical restriction theory is driven by the interaction between oscillation and the geometry of the underlying measure, and it is closely connected with Bochner--Riesz problems, geometric measure theory, and
dispersive partial differential equations; see, for instance, \cite{F70,S86,T75,T99}. For curves in the plane with non-vanishing curvature, the sharp restriction theorem goes back to Zygmund \cite{Z74}. For smooth
hypersurfaces with non-vanishing Gaussian curvature, the fundamental $L^2$-restriction estimate is given by the Stein--Tomas theorem \cite{T75,S86}. Sharp results for finite-type hypersurfaces in three dimensions were obtained by Ikromov, Kempe and M\"uller, and by Ikromov and M\"uller, using Newton polyhedra and uniform oscillatory integral estimates; see \cite{IKM10,IM11,IM16}. Restriction and extension estimates are also closely related to Strichartz estimates and pointwise convergence problems for dispersive equations \cite{ST77,MVV96}. Beyond smooth hypersurfaces, the theory has been developed for fractal measures \cite{M00}, while more recent advances have been driven by multilinear methods, polynomial partitioning, and decoupling \cite{BCT06,G16,G18,D20}.

In the phase-space setting, the classical extension operator admits a natural operator-valued analogue.

Indeed, if $\mu$ is a compactly supported measure on $\RR^{2d}$, define the \emph{symplectic Fourier extension operator} by
\begin{equation}\label{Def-symplecticextension}
    \mathcal E_\sigma G=\SF(G\mu)
\end{equation}
and the \emph{operator-valued extension operator} by
\[
    \mathcal E_WG=\rho(G\mu)=L_{\SF(G\mu)}=L_{\mathcal E_\sigma G}.
\]
The corresponding \emph{Schatten-class restriction estimate} is
\begin{equation}\label{Est-Weylextension}
    \|L_{\mathcal E_\sigma G}\|_{\Scal_{p'}}\lesssim \|G\|_{L^{q'}(\mu)}.
\end{equation}
Luef and Samuelsen proved that \eqref{Est-Weylextension} holds if and only if its classical counterpart
\begin{equation}\label{Est-symplecticextension}
    \|\mathcal E_\sigma G\|_{L^{p'}(\RR^{2d})} \lesssim \|G\|_{L^{q'}(\mu)}
\end{equation}
holds; see \cite[Theorem~1.1]{LS25}. Since the symplectic phase differs from the Euclidean Fourier phase only by a linear change of variables, \eqref{Est-symplecticextension} is equivalent to \eqref{Est-classicalextension} with $n=2d$. Their proof uses Gaussian windows and the Cohen class of an operator. If $\operatorname{supp}\mu$ is contained in a ball of radius $R$, the resulting equivalence constants have exponential dependence of the form $\ee^{cR^2}$; see
\cite[Proof of Theorem~1.1]{LS25}.

M\"uller subsequently gave a more structural version of this equivalence \cite{M26}. If $u\in\mathcal E'(\RR^{2d})$ is supported in a ball of radius $R$, then
\[
    \|L_{\SF u}\|_{\Scal_p}\simeq_R
    \|\SF u\|_{L^p(\RR^{2d})},
    \,\,\, 1\leq p\leq\infty.
\]
This extends the result of Luef and Samuelsen from measures to arbitrary compactly supported distributions and replaces the exponential dependence by a polynomial one. More precisely, for $R\geq1$, the loss obtained in \cite[Theorem~4.2 and Remark~4.3]{M26} is of order \(R^{(5d+2)|1-2/p|}.\)

Related works of Mishra and Vemuri treat Weyl transforms of measures, smooth measures on real-analytic submanifolds, and compactly supported distributions; see \cite{MV23,MV24,MV25}. In particular, \cite[Theorem~1]{MV24} proves a two-sided comparison for compactly supported distributions, with constants depending on the support, but without making the dependence on its size explicit.

Our first result sharpens the polynomial dependence in M\"uller's result in \cite{M26}.

\begin{theorem}\label{Mainthm}
Let $u\in\mathcal E'(\RR^{2d})$ and assume that \(\operatorname{supp}u\subset\overline{B(z_0,R)}\)
for some $z_0\in\RR^{2d}$ and $R\geq1$. Then, for every
$1\leq p\leq\infty$ and every $\varepsilon>0$,
\begin{equation}\label{Maininequality1}
    \|\SF u\|_{L^p(\RR^{2d})} \lesssim_{d,p,\varepsilon} R^{(2d+1+\varepsilon)\left|1-\frac{2}{p}\right|} \|L_{\SF u}\|_{\Scal_p},
\end{equation}
and
\begin{equation}\label{Maininequality2}
    \|L_{\SF u}\|_{\Scal_p}\lesssim_{d,p,\varepsilon} R^{(2d+1+\varepsilon)\left|1-\frac{2}{p}\right|}\|\SF u\|_{L^p(\RR^{2d})}.
\end{equation}
The implicit constants are independent of $u$, $z_0$, and $R$.
\end{theorem}

Thus the exponent $(5d+2)|1-2/p|$ is replaced by $(2d+1+\varepsilon)|1-2/p|$. In particular, at the trace-class and operator-norm endpoints the loss improves from $R^{5d+2}$ to $R^{2d+1+\varepsilon}$. At $p=2$ the power vanishes, in agreement with the exact Hilbert--Schmidt identity in Theorem \ref{Thm-schmidthilbert}.

The main improvement occurs at the trace-class endpoint. M\"uller's argument uses the full Hermite expansion; see \cite[Lemma~3.2]{M26}. We instead exploit the radial Hermite--Laguerre correspondence
\[
    \rho(\varphi_k)=P_k,\,\,\, \varphi_k(z)\defeq L_k^{d-1}\!\left(\frac{|z|^2}{2}\right)\ee^{-|z|^2/4},
\]
where $P_k$ is the orthogonal projection onto the $k$-th Hermite eigenspace. For a radial function, the Weyl operator is therefore diagonal with respect to the Hermite decomposition, and its trace norm reduces to a weighted sum of the ranks of the finite-rank projections $P_k$. This yields the sharper radial estimate in Lemma~\ref{Lem-S1normN2norm}. Interpolation with the Hilbert--Schmidt identity in Theorem~\ref{Thm-schmidthilbert} gives the regularized comparison in Theorem~\ref{Thm-oneway}. The opposite direction is then obtained through Schatten duality and Werner convolution in Theorem~\ref{Thm-dualityPsi}, which differs from M\"uller's approach in \cite{M26}, where the Calder\'on--Vaillancourt theorem is used to provide a separate operator-norm endpoint before interpolation. The relevant Hermite and Laguerre background is recalled in Section~\ref{Subsection-backgroundhermitelaguerre}.

Our second objective is to determine whether the dependence on $R$ is merely a feature of the proof or a genuine obstruction. Before the present work, the available evidence was essentially restricted to the operator-norm endpoint: \cite[Example~4.4]{M26} gives a nontrivial lower bound for
\eqref{Maininequality2} at $p=\infty$. We construct compactly supported examples showing polynomial growth of the best constants in substantially larger ranges of Schatten exponents, and also obtain lower bounds for the reverse comparison \eqref{Maininequality1}.

For $R\geq1$, define
\begin{equation}\label{Def-CpR}
\begin{split}
C_{d,p}(R)\defeq\sup\bigg\{\frac{\|L_{\SF u}\|_{\Scal_p}}   {\|\SF u\|_{L^p(\RR^{2d})}}: u\in\mathcal E'(\RR^{2d}),\,\operatorname{supp}u\subset\overline{B(0,2R)},\, \SF u\neq0\bigg\},
\end{split}
\end{equation}
and
\begin{equation}\label{Def-DpR}
\begin{split}
D_{d,p}(R)\defeq\sup\bigg\{\frac{\|\SF u\|_{L^p(\RR^{2d})}}  {\|L_{\SF u}\|_{\Scal_p}}: u\in\mathcal E'(\RR^{2d}),\,\operatorname{supp}u\subset\overline{B(0,2R)}, \,0<\|L_{\SF u}\|_{\Scal_p}<\infty\bigg\}.
\end{split}
\end{equation}
We measure their polynomial growth by
\begin{equation}\label{Def-gammaeta}
\gamma_{d,p}\defeq\inf\left\{s\geq0:\sup_{R\geq1}\frac{C_{d,p}(R)}{R^s}<\infty\right\},\,\,\,
\eta_{d,p}\defeq\inf\left\{s\geq0:\sup_{R\geq1}\frac{D_{d,p}(R)}{R^s}<\infty\right\}.
\end{equation}
Since the estimates in Theorem~\ref{Mainthm} hold for every \(\varepsilon>0\), the definition of the growth exponents implies
\[
    \gamma_{d,p},\eta_{d,p} \leq(2d+1)\left|1-\frac2p\right|\eqdef  U_{d,p}.
\]
The lower bounds are obtained from truncated Laguerre functions in Proposition \ref{Prop-explicitexample}. The identity $\rho(\varphi_k)=P_k$ makes the Schatten norm explicit, while known
$L^p$-asymptotics for $\varphi_k$ determine the size of the phase-space norm. After cutting off at radius comparable to $\sqrt{k}$, the discarded tail is exponentially small. This leads to the following result.

\begin{theorem}\label{Thm-nontrivial}
Assume $d\geq2$ and let \(p_c\defeq\frac{4d}{2d-1}\).
For every $1\leq p\leq\infty$ with $p\neq p_c$,
\begin{equation}\label{For-etagamma}
\begin{aligned}
    \underline{\gamma_{d,p}}\defeq\max\{a_{d,p'},-a_{d,p},0\}
    &\leq\gamma_{d,p}\leq U_{d,p},\\
    \underline{\eta_{d,p}}\defeq\max\{a_{d,p},0\}
    &\leq\eta_{d,p}\leq U_{d,p},
\end{aligned}
\end{equation}
where
\begin{equation}\label{Def-ap}
    a_{d,p}\defeq
    \begin{cases}
    2(d-1)-\dfrac{2(2d-1)}{p},
        & p_c<p\leq\infty,\\[6pt]
    \dfrac{2}{p}-1,
        & 1\leq p<p_c.
    \end{cases}
\end{equation}
\end{theorem}

For completeness, Figure~\ref{fig:growth-exponents} compares the lower and upper exponents when $d=2$. 

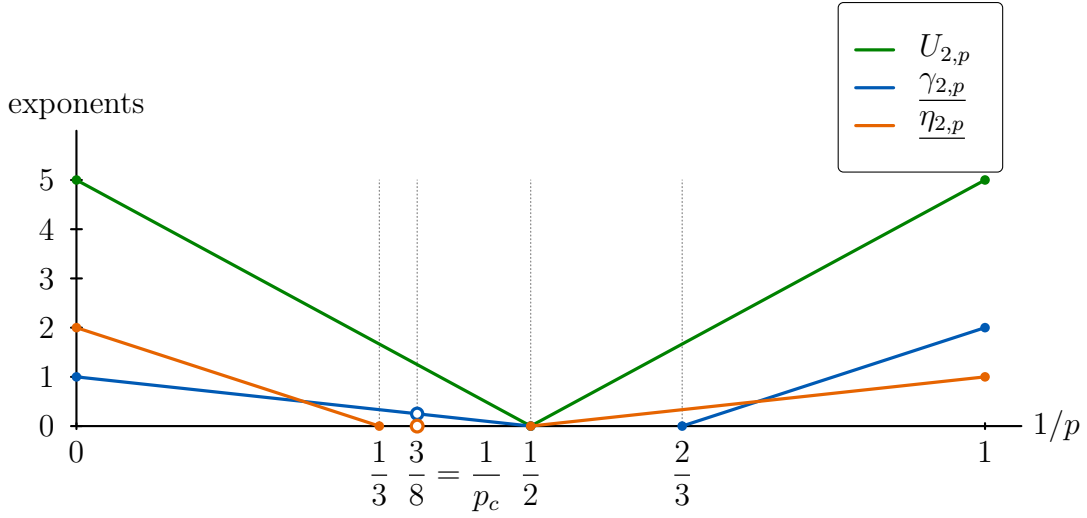
\begin{figure}[htbp]
\centering
\begin{tikzpicture}[
    x=12cm,
    y=0.65cm,
    line cap=round,
    line join=round,
    >=stealth
]
\pgfmathsetmacro{\dval}{2}
\pgfmathsetmacro{\xe}{(\dval-1)/(2*\dval-1)}
\pgfmathsetmacro{\xc}{(2*\dval-1)/(4*\dval)}
\pgfmathsetmacro{\xh}{0.5}
\pgfmathsetmacro{\xg}{\dval/(2*\dval-1)}
\pgfmathsetmacro{\Umax}{2*\dval+1}
\pgfmathsetmacro{\etaL}{2*(\dval-1)}
\pgfmathsetmacro{\gammaR}{2*(\dval-1)}
\pgfmathsetmacro{\gammaAtPc}{1-2*\xc}

\definecolor{gammaColor}{RGB}{0,90,180}
\definecolor{etaColor}{RGB}{230,100,0}
\definecolor{UColor}{RGB}{0,130,0}

\draw[black,thick] (0,0) -- (1.04,0) node[right] {$1/p$};
\draw[black,thick] (0,0) -- (0,\Umax+1) node[above] {exponents};

\foreach \xx in {\xe,\xc,\xh,\xg}{
    \draw[black!55,densely dotted] (\xx,0) -- (\xx,\Umax);
}

\foreach \xx/\lab/\xshift in {
    0/{$0$}/0pt,
    \xe/{$\dfrac13$}/0pt,
    \xc/{$\dfrac38=\dfrac{1}{p_c}$}/14pt,
    \xh/{$\dfrac12$}/0pt,
    \xg/{$\dfrac23$}/0pt,
    1/{$1$}/0pt
}{
    \draw[black,thick] (\xx,0.06) -- (\xx,-0.06);
    \node[below,xshift=\xshift] at (\xx,-0.08) {\lab};
}

\foreach \yy in {0,1,2,3,4,5}{
    \draw[black,thick] (0.006,\yy) -- (-0.006,\yy);
    \node[left] at (-0.012,\yy) {$\yy$};
}

\draw[UColor,very thick] (0,\Umax) -- (\xh,0) -- (1,\Umax);
\fill[UColor] (0,\Umax) circle (1.8pt);
\fill[UColor] (\xh,0) circle (1.8pt);
\fill[UColor] (1,\Umax) circle (1.8pt);

\draw[gammaColor,very thick] (0,1) -- (\xh,0);
\draw[gammaColor,very thick] (\xg,0) -- (1,\gammaR);
\fill[gammaColor] (0,1) circle (1.8pt);
\fill[gammaColor] (\xh,0) circle (1.8pt);
\fill[gammaColor] (\xg,0) circle (1.8pt);
\fill[gammaColor] (1,\gammaR) circle (1.8pt);
\draw[gammaColor,fill=white,very thick]
    (\xc,\gammaAtPc) circle (2.1pt);

\draw[etaColor,very thick] (0,\etaL) -- (\xe,0);
\draw[etaColor,very thick] (\xh,0) -- (1,1);
\fill[etaColor] (0,\etaL) circle (1.8pt);
\fill[etaColor] (\xe,0) circle (1.8pt);
\fill[etaColor] (\xh,0) circle (1.8pt);
\fill[etaColor] (1,1) circle (1.8pt);
\draw[etaColor,fill=white,very thick] (\xc,0) circle (2.2pt);

\node[
    draw=black,
    rounded corners=2pt,
    fill=white,
    inner sep=6pt,
    anchor=south east
]
at (1.02,\Umax+0.15) {
\begin{tikzpicture}[x=1cm,y=1cm,baseline={(current bounding box.center)}]
    \draw[UColor,very thick] (0,1.00) -- (0.50,1.00);
    \node[right] at (0.62,1.00) {$U_{2,p}$};

    \draw[gammaColor,very thick] (0,0.50) -- (0.50,0.50);
    \node[right] at (0.62,0.50) {$\underline{\gamma_{2,p}}$};

    \draw[etaColor,very thick] (0,0.00) -- (0.50,0.00);
    \node[right] at (0.62,0.00) {$\underline{\eta_{2,p}}$};
\end{tikzpicture}
};
\end{tikzpicture}
\caption{Lower and upper polynomial exponents as functions of $1/p$.}
\label{fig:growth-exponents}
\end{figure}

The lower bounds in Theorem~\ref{Thm-nontrivial} show that the support-radius
loss cannot in general be removed.

\begin{corollary}\label{Cor-nontrivial}
Assume $d\geq2$.
\begin{enumerate}
\item For
\begin{equation}\label{Range2}
    p\in[1,2)\cup\left(2+\frac{1}{d-1},\infty\right],
\end{equation}
there is no constant $C$ independent of $R$ such that
\[
    \|\SF u\|_{L^p(\RR^{2d})}
    \leq C\|L_{\SF u}\|_{\Scal_p},
    \,\,\,
    \operatorname{supp}u\subset\overline{B(z_0,R)}.
\]
\item For
\begin{equation}\label{Range1}
    p\in\left[1,2-\frac1d\right)\cup(2,\infty],
\end{equation}
there is no constant $C$ independent of $R$ such that
\[
    \|L_{\SF u}\|_{\Scal_p}
    \leq C\|\SF u\|_{L^p(\RR^{2d})},
    \,\,\,
    \operatorname{supp}u\subset\overline{B(z_0,R)}.
\]
\end{enumerate}
\end{corollary}

Thus the dependence on $R$ in Theorem~\ref{Mainthm} reflects a genuine obstruction, rather than merely a limitation of the proof. Consequently, the estimate from \(\Scal_p\) to \(L^p\) cannot hold globally for \(p\) in \eqref{Range2}, while the estimate from \(L^p\) to \(\Scal_p\) cannot hold globally for \(p\) in \eqref{Range1}. In particular, in these ranges there can be no corresponding comparison estimate without a Fourier-support restriction.

The construction refines \cite[Example~4.4]{M26}, which provides a nontrivial lower bound only for \eqref{Maininequality2} at $p=\infty$. As $d\to\infty$, the ranges \eqref{Range2} and \eqref{Range1} approach the full interval $1\leq p\leq\infty$, apart from the Hilbert--Schmidt point $p=2$. Furthermore, for every fixed $p>2$,
\[
    \lim_{d\to\infty}
    \frac{U_{d,p}}{\underline{\eta_{d,p}}}=1.
\]
Hence the exponent in the comparison from $\Scal_p$ to $L^p$, namely \eqref{Maininequality1}, is asymptotically sharp in high dimensions.

The critical exponent $p=p_c$ is omitted from Theorem~\ref{Thm-nontrivial} because the corresponding $L^p$-asymptotics for Laguerre functions contain an additional logarithmic factor. We do not pursue
that endpoint refinement here. The upper bounds in Theorem~\ref{Mainthm} hold for every $d\geq1$, whereas the lower-bound theorem is stated for $d\geq2$.

%The paper is organized as follows. In Section~\ref{Section-pre} we recall the Fourier--Weyl formalism, Schatten classes, and Werner's convolution product. Section~\ref{Section-Laguerre} reviews the Hermite--Laguerre correspondence and collects the Laguerre estimates used in the examples. In Section~\ref{Section-Mainestimate} we prove the refined comparison estimates with explicit dependence on the support radius. Finally, in Section~\ref{Section-example} we construct the truncated Laguerre examples and prove the lower bounds above.

\subsection*{Notation}

Throughout the paper, $d\geq1$ is fixed. If $A$ and $B$ are nonnegative quantities, we write $A\lesssim B$ if there exists a constant $C>0$ such that $A\leq CB$, and $A\simeq B$ if both $A\lesssim B$ and $B\lesssim A$ hold. Dependence of the implicit constant on parameters is indicated by subscripts; for example, $A\lesssim_{d,p}B$.

For $1\leq p\leq\infty$, $p'$ denotes the conjugate exponent. We write $\Scal(\RR^n)$ for the Schwartz space, $\Scal'(\RR^n)$ for the space of tempered distributions, and $\mathcal E'(\RR^n)$ for the space of compactly supported distributions. The parameter $R\geq1$ always denotes the support scale.

We write $x\cdot y$ for the Euclidean inner product on $\RR^n$. The Hilbert space inner product on $L^2(\RR^n)$ is
\[
    (f,g)_{L^2}=\int_{\RR^n}f(x)\overline{g(x)}\,\dd x,
\]
and is linear in the first variable. The distributional pairing between $\Scal'(\RR^n)$ and $\Scal(\RR^n)$ is denoted by $\langle u,\varphi\rangle$; for suitable functions,
\[
    \langle f,g\rangle =\int_{\RR^n}f(x)g(x)\,\dd x.
\]

\section{Fourier--Weyl preliminaries and Schatten classes}\label{Section-pre}
Although the main objects have already been introduced in Section~\ref{Section-intro}, we record here the precise normalizations and basic identities used throughout the paper.

\subsection{Fourier transform and symplectic Fourier transform}
Throughout the paper we use the following convention for the Euclidean Fourier transform on \(\RR^d\):
\[
\widehat f(\xi)=\int_{\RR^d} f(x)\ee^{-ix\cdot \xi}\,\dd x, \,\,\, f(x)=(2\pi)^{-d}\int_{\RR^d}\widehat f(\xi)\ee^{ix\cdot \xi}\,\dd\xi .
\]
With this convention, Plancherel's theorem reads
\[
\|\widehat f\|_{L^2(\RR^d)}=(2\pi)^{d/2}\|f\|_{L^2(\RR^d)}.
\]

On the phase space \(\RR^{2d}\), for \(z=(x,\xi)\) and \(w=(y,\eta)\), we use the \emph{symplectic form}
\[
\sigma(w,z)=\eta\cdot x-y\cdot \xi. 
\]
Equivalently, if 
\[
J= \begin{pmatrix} 0 & \Id\\ -\Id & 0 \end{pmatrix}, 
\] 
then
\begin{equation}\label{Def-MatrixJ}
\sigma(w,z)=Jw\cdot z,\,\,\, z,w\in\RR^{2d}.
\end{equation}

The \emph{normalized symplectic Fourier transform} is
\begin{equation}\label{Def-Sym}
\SF f(w)\defeq (2\pi)^{-d}\int_{\RR^{2d}} f(z)\ee^{-i\sigma(w,z)}\,\dd z,\,\,\,w\in \RR^{2d}.
\end{equation}
Because of (\ref{Def-MatrixJ}),
\begin{equation}\label{Rela-twotransforms}
\SF f(w)=(2\pi)^{-d}\widehat f(Jw),
\end{equation}
where the Fourier transform on the right-hand side is the Euclidean Fourier transform on \(\RR^{2d}\). Since \(J\) is orthogonal, \(\SF\) is unitary on \(L^2(\RR^{2d})\), and
\[
\SF^2=\mathrm{Id}.
\]
We use the standard \emph{convolution}
\[
    (f*g)(z)=\int_{\RR^{2d}}f(z-z')g(z')\,\dd z'.
\]
With our normalization of the symplectic Fourier transform in \eqref{Def-Sym}, one has
\begin{equation}\label{Conv-sym}
    \SF(FG)=(2\pi)^{-d}\,\SF F*\SF G.
\end{equation}

\subsection{The integrated Schr\"odinger representation}
We use the \emph{Schr\"odinger representation} of phase space, that is, the map \(\rho:\RR^{2d}\to \{\mathrm{unitary}\,\,\mathrm{ operators}\,\,\mathrm{on}\,\, L^2(\RR^d)\}\) given by 
\[
[\rho(x,\xi)g](t) =\ee^{-\frac{i}{2}x\cdot \xi}\ee^{it\cdot \xi}g(t-x), \,\,\, g\in L^2(\RR^d). 
\] It satisfies the projective relation 
\begin{equation}\label{Idn-rhourhov}
\rho(z)\rho(z') = \ee^{\frac i2\sigma(z,z')}\rho(z+z'). 
\end{equation}
For \(F\in L^1(\RR^{2d})\), the \emph{integrated representation} is defined by 
\begin{equation}\label{Def-integratedSch}
\rho(F)g = (2\pi)^{-d} \int_{\RR^{2d}}F(z)\rho(z)g\,\dd z, \,\,\, g\in L^2(\RR^d), 
\end{equation} 
where the integral is understood as a Bochner integral. So for any \(F\in L^1(\RR^{2d})\), \(\rho(F)\in \BB(L^2(\RR^d))\). The following kernel formula will be used repeatedly. Writing \(z=(y,\xi)\), one has
\[
[\rho(F)g](t) = (2\pi)^{-d} \int_{\RR^d}\int_{\RR^d} F(y,\xi)\ee^{-\frac{i}{2}y\cdot \xi}\ee^{it\cdot \xi}g(t-y) \,\dd y\,\dd \xi . 
\]
After the change of variables \(x=t-y\), this becomes 
\[
[\rho(F)g](t)=\int_{\RR^d} k_F(t,x)g(x)\,\dd x, 
\] 
where 
\[
k_F(t,x) = (2\pi)^{-d}\int_{\RR^d} F(t-x,\xi)\ee^{\frac{i}{2}(t+x)\cdot \xi}\,\dd \xi. 
\] 
Equivalently, if \(\mathcal F_2\) denotes the Euclidean Fourier transform in the second variable only, namely 
\[
\mathcal F_2 F(y,\eta) = \int_{\RR^d}F(y,\xi)\ee^{-i\eta\cdot\xi}\,\dd\xi, 
\] 
then 
\begin{equation}\label{For-kernelrhoF1}
k_F(t,x) = (2\pi)^{-d} \mathcal F_2 F \left( t-x,-\frac{t+x}{2} \right). 
\end{equation}

For a suitable phase-space symbol \(F\), for instance \(F\in\Scal(\RR^{2d})\), we define its \emph{Weyl operator} by
\[
L_F\defeq \rho(\SF F).
\]
That is, 
\[
L_F = (2\pi)^{-d} \int_{\RR^{2d}} \SF F(z)\rho(z)\,\dd z . 
\] 
Since \(\SF^2=\Id\), we also have 
\[ 
\rho(F)=L_{\SF F} 
\] 
for suitable functions \(F\), and more generally for suitable distributions on \(\RR^{2d}\).

%In particular, if \(F\in\Scal(\RR^{2d})\), then \(\rho(F)\) is of trace class. Its trace is given by
%\[
%\tr \rho(F)=\int_{\RR^d} k_F(x,x)\,\dd x.
%\]
%Using the kernel formula \eqref{For-kernelrhoF1}, we obtain  
%\[ 
%k_F(x,x) = (2\pi)^{-d} \int_{\RR^d} F(0,\xi)\ee^{ix\cdot\xi}\,\dd\xi.
%\]
%For \(F\in\Scal(\RR^{2d})\), Fourier inversion gives
%\[
%\tr \rho(F)=F(0,0).
%\]
Finally, we extend the construction to tempered distributions. If \(u\in\Scal^{\prime}(\RR^{2d})\), then \(\rho(u)\) is understood as a continuous linear operator 
\[
\rho(u):\Scal(\RR^d)\to \Scal^{\prime}(\RR^d). 
\] 
Its Schwartz kernel \(k_u\) is defined by 
\begin{equation}\label{For-kernelrhou} 
k_u(t,x) = (2\pi)^{-d} (\mathcal F_2 u) \left( t-x,-\frac{t+x}{2} \right), 
\end{equation}
where the linear change of variables is interpreted in the sense of distributions. Equivalently, in the sense of distributions, for all \(\phi,\psi\in\Scal(\RR^d)\),
\begin{equation}\label{For-kernelfordistribution}
\langle \rho(u)\psi,\phi\rangle = \langle k_u,\phi\otimes\psi\rangle . 
\end{equation}
This distributional formulation will be used later when the phase-space symbol is no longer a Schwartz function.
\begin{remark}\label{Rem-Linjective}
The map \( u\mapsto k_u\) is a topological isomorphism of \(\Scal'(\RR^{2d})\). Indeed, by \eqref{For-kernelrhou}, it is obtained from the partial Fourier transform in the second variable and an invertible linear change of variables. Hence, by the Schwartz kernel theorem, the correspondence \(u\mapsto \rho(u)\) is one-to-one at the level of tempered distributions, after identifying operators \(\Scal(\RR^d)\to\Scal'(\RR^d)\) with their Schwartz kernels. 

Thus, the Weyl correspondence \(u\mapsto L_u=\rho(\SF u)\)
is a topological isomorphism
\[
    \Scal'(\RR^{2d}) \simeq\mathcal L\bigl(\Scal(\RR^d),\Scal'(\RR^d)\bigr).
\]
\end{remark}
We shall also use the twisted convolution associated with the projective Schr\"odinger representation. Let \(F,G\) be suitable functions on \(\RR^{2d}\), for instance \(F,G\in\Scal(\RR^{2d})\). We define the \emph{twisted convolution} between \(F\) and \(G\) as
\begin{equation}\label{Def-twistedconvolution}
(F\times G)(z)\defeq(2\pi)^{-d}\int_{\RR^{2d}}F(z-z')G(z')\ee^{\frac i2\sigma(z,z')}\,\dd z'.
\end{equation}
This definition is chosen so that the integrated representation transforms the twisted convolution into operator composition. Indeed, using \eqref{Idn-rhourhov}, we compute formally
\[
\begin{aligned}
\rho(F)\rho(G)
&=(2\pi)^{-2d}\int_{\RR^{2d}}\int_{\RR^{2d}}F(u)G(v)\rho(u)\rho(v)\,\dd u\,\dd v  \\
&=(2\pi)^{-2d}\int_{\RR^{2d}}\int_{\RR^{2d}}F(u)G(v)\ee^{\frac i2\sigma(u,v)}\rho(u+v)\,\dd u\,\dd v .
\end{aligned}
\]
Putting \(z=u+v\), or equivalently \(u=z-v\), we obtain
\begin{equation}\label{Idn-rhoFrhoG}
\rho(F)\rho(G)=(2\pi)^{-d}\int_{\RR^{2d}}(F\times G)(z)\rho(z)\,\dd z=\rho(F\times G).
\end{equation}

\subsection{Schatten classes}

We also recall some basic facts about Schatten classes. Let \(\mathcal B(L^2(\RR^d))\) denote the space of the bounded operators on \(L^2(\RR^d)\), and let \(\mathcal K(L^2(\RR^d))\) denote the space of the compact operators.

By the singular value decomposition in \cite[Chapter 3]{S15}, every \(A\in\mathcal K(L^2(\RR^d))\) can be written as
\[
 A=\sum_{n\ge1}s_n(A)(\,\cdot\,,g_n)_{L^2}f_n,
\]
where \(s_1(A)\ge s_2(A)\ge \cdots\ge0\) are the singular values of \(A\), namely the eigenvalues of \((A^*A)^{1/2}\), counted with multiplicity, and \(\{f_n\}\), \(\{g_n\}\) are orthonormal systems in \(L^2(\RR^d)\). 

For \(1\le p<\infty\), the \emph{Schatten class} \(\Scal_p\) consists of all compact operators \(A\) such that
\[
\|A\|_{\Scal_p}\defeq\left(\sum_{j\ge1}s_j(A)^p\right)^{1/p}<\infty.
\]
For \(p=\infty\), we set
\[
\Scal_\infty=\BB(L^2(\RR^d)),
\,\,\,
\|A\|_{\Scal_\infty}=\|A\|_{\BB(L^2(\RR^d))}.
\]
The case \(p=2\) is the \emph{Hilbert--Schmidt class}. If \(A\) is an integral operator on \(L^2(\RR^d)\) with kernel \(K_A(t,x)\), then
\[
    \|A\|_{\Scal_2}^2=\int_{\RR^d}\int_{\RR^d}|K_A(t,x)|^2\,\dd t\,\dd x.
\]
The case \(p=1\) is the \emph{trace class}. If \(A\in\Scal_1\), then
\[
    \tr(A)=\sum_{j\ge1}(Ae_j,e_j),
\]
where \(\{e_j\}_{j\ge1}\) is any orthonormal basis of \(L^2(\RR^d)\). For integral operators with sufficiently regular kernels, this trace is given by
\[
    \tr(A)=\int_{\RR^d}K_A(x,x)\,\dd x.
\]

Moreover, we have the continuous inclusions
\begin{equation}\label{Rela-Inclusion}
\Scal_1\subseteq \Scal_p \subseteq \Scal_q \subseteq \Scal_\infty,\,\,\,1\leq p\leq q\leq \infty.
\end{equation}
These inclusions follow from the corresponding inclusions for sequence spaces. In addition, finite-rank operators are dense in \(\Scal_p\) for \(1\le p<\infty\); in particular, \(\Scal_1\) is dense in \(\Scal_p\) for \(1\le p<\infty\). Moreover, we recall that \(\Scal_p\) is a two-sided ideal in \(\Scal_\infty=B(L^2(\RR^d))\) and is closed under taking adjoints. In particular, if \(A\in\Scal_p\) and  \(B\in\Scal_\infty\), then \(A^*,AB, BA\in\Scal_p\). 

We shall use repeatedly the trace duality of Schatten classes. If \(1\le p<\infty\), then every \(B\in\Scal_{p'}\) defines a bounded linear functional on \(\Scal_p\) by
\[
    A\mapsto \tr(AB^*),
\]
and
\[
    |\tr(AB^*)| \le\|A\|_{\Scal_p}\|B\|_{\Scal_{p'}}.
\]
Moreover, this identifies \((\Scal_p)^*\) with \(\Scal_{p'}\). For \(p=2\), we use the Hilbert--Schmidt inner product
\[
    (A,B)_{\Scal_2}:=\tr(AB^*),\qquad A,B\in\Scal_2.
\]

We shall also use the interpolation property of Schatten classes, which can be found in \cite{BL76, S05, S15}. If
\(1\le p_0,p_1\le\infty\), \(0<\theta<1\), and
\[
\frac1p=\frac{1-\theta}{p_0}+\frac{\theta}{p_1},
\]
then
\[
[\mathcal S_{p_0},\mathcal S_{p_1}]_\theta=\mathcal S_p
\]
with equality of norms up to the usual interpolation constants, here \([\cdot,\cdot]_{\theta}\) denotes the complex interpolation functor. This allows us
to obtain estimates in \(\mathcal S_p\) by interpolating between the trace class,
the Hilbert--Schmidt class, and the operator norm.

We also recall \emph{Werner's convolution product} for Schatten class operators. For
\(T_1,T_2\in \Scal_1\), define
\[
T_1\star T_2(w)\defeq\tr\bigl[\rho(-w)T_1\rho(w)\,P T_2P\bigr],\,\,\, w\in\RR^{2d},
\]
where \(P\) is the parity operator on \(L^2(\RR^d)\), given by
\[
Pg(t)\defeq g(-t),\,\,\, g\in L^2(\RR^d).
\]
This convolution product is commutative and associative; see
\cite[Proposition~4.4]{LS18}.

We record the following identities, which follow from \cite[(2.9) and Lemma~3.4]{M26}. According to \eqref{Conv-sym}, we have additional \((2\pi)^{-d}\) as follows.
\begin{lemma}\label{Lem-WernerConvolution}
Let \(F,G\in\Scal(\RR^{2d})\). Then
\[
\rho(F)\star\rho(G)=\SF(FG)=(2\pi)^{-d}\SF F * \SF G .
\]
More generally, if \(u\in\Scal^{\prime}(\RR^{2d})\), \(T=\rho(u)\in\Scal_\infty\), and
\(F\in\Scal(\RR^{2d})\), then
\begin{equation}\label{Oper-star}
T\star\rho(F)=\rho(u)\star\rho(F)=\SF(uF)=(2\pi)^{-d}\SF u * \SF F .
\end{equation}
\end{lemma}

The following \emph{Young-type inequality for Werner convolution} will be used below; see, for instance, \cite[Proposition~3.2]{W84}, \cite[Proposition~4.2]{LS18}, and \cite[Proposition~2.1]{M26}.

\begin{proposition}\label{Prop-Younginequality}
Let \(1\le p,q,r\le\infty\) satisfy
\[
\frac1p+\frac1q=1+\frac1r.
\]
Then, for \(S\in\Scal_p\) and \(T\in\Scal_q\), the convolution
\(S\star T\) belongs to \(L^r(\RR^{2d})\), and
\[
\|S\star T\|_{L^r(\RR^{2d})}\lesssim\|S\|_{\Scal_p}\|T\|_{\Scal_q}.
\]
\end{proposition}

These preliminaries provide the analytic framework for comparing Schatten norms of Weyl operators with \(L^p(\RR^{2d})\)-norms of their phase-space symbols. In the next section we recall the Hermite--Laguerre correspondence, used in the proof of the main estimates in Section~\ref{Section-Mainestimate}, together with the Laguerre estimates needed for the examples in Section~\ref{Section-example}.

\section{Hermite--Laguerre correspondence and auxiliary estimates}\label{Section-Laguerre}
\subsection{Hermite projections and Laguerre functions}\label{Subsection-backgroundhermitelaguerre}
We first recall the Hermite decomposition of \(L^2(\RR^d)\). Let
\[
H_1=-\frac{\dd^2}{\dd x^2}+x^2
\]
be the \emph{one-dimensional Hermite operator}. For \(k\in\NN\), define the \emph{normalized Hermite polynomial}:
\[
h_k(x)\defeq\frac{(-1)^k\ee^{x^2/2}}{(2^k k!\sqrt{\pi})^{1/2}}\frac{\dd^k}{\dd x^k}\ee^{-x^2}.
\]
Then \(\{h_k\}_{k\in\NN}\) is a complete orthonormal basis of \(L^2(\RR)\), and
\[
H_1h_k=(2k+1)h_k.
\]

In dimension \(d\), for a multi-index
\[
\mathbf k=(k_1,\dots,k_d)\in\NN^d,
\]
we define
\[
h_{\mathbf k}(x)\defeq h_{k_1}(x_1)\cdots h_{k_d}(x_d),\,\,\, x=(x_1,\dots,x_d)\in\RR^d.
\]
The family \(\{h_{\mathbf k}\}_{\mathbf k\in\NN^d}\) is a complete orthonormal basis of \(L^2(\RR^d)\). Moreover, if
\[
H_d=-\Delta+|x|^2=\sum_{j=1}^d\left(-\frac{\partial^2}{\partial x_j^2}+x_j^2\right),
\]
then
\begin{equation}\label{Hermitebasis}
H_dh_{\mathbf k}=(d+2|\mathbf k|)h_{\mathbf k},\,\,\,|\mathbf k|=k_1+\cdots+k_d.
\end{equation}
Thus the eigenspaces of \(H_d\) are indexed by the integer \(k=|\mathbf k|\). This motivates the definition of the \emph{orthogonal projection} onto the \(k\)-th Hermite eigenspace:
\[
P_k f
\defeq\sum_{|\mathbf k|=k}(f,h_{\mathbf k})_{L^2(\RR^d)}h_{\mathbf k}.
\]
 Equivalently, every \(f\in L^2(\RR^d)\) admits the decomposition
\[
f=\sum_{k=0}^{\infty}P_k f.
\]
Here \(P_k f\) is orthogonal projection of \(f\) on the eigenspace 
\[
\mathcal V_k\defeq\operatorname{span}\{h_{\mathbf k}:\mathbf k\in\NN^d,\ |\mathbf k|=k\},
\]
which corresponds to the eigenvalue \(d+2k\).
We now explain how these Hermite projections appear from radial functions on the phase space \(\RR^{2d}\). For multi-indices \(\mathbf k,\mathbf j\in\NN^d\), define the matrix coefficient
\[
h_{\mathbf k\mathbf j}(z)\defeq (\rho(z)h_{\mathbf k},h_{\mathbf j})_{L^2(\RR^d)},\,\,\,z\in\RR^{2d}.
\]
These functions live on the phase space \(\RR^{2d}\). In particular, the diagonal coefficients
\[
h_{\mathbf k\mathbf k}(z)=(\rho(z)h_{\mathbf k},h_{\mathbf k})_{L^2(\RR^d)}
\]
encode the action of the phase-space shift \(\rho(z)\) on the Hermite state \(h_{\mathbf k}\).

Summing these diagonal coefficients over all Hermite functions with the same energy level \(|\mathbf k|=k\) gives a radial function on \(\RR^{2d}\). More precisely, for \(\alpha>-1\), \(L_k^\alpha\) denotes the \emph{generalized Laguerre polynomial} of degree \(k\) and type \(\alpha\):
\[
L_k^\alpha(t)\defeq\frac{t^{-\alpha}\ee^t}{k!}\frac{\dd^k}{\dd t^k}\left(\ee^{-t}t^{k+\alpha}\right),\,\,\, t> 0,\,\, k\in\NN,
\]
which is explicitly given by
\[
L_k^\alpha(t)=\sum_{j=0}^{k}\frac{\Gamma(k+\alpha+1)}{\Gamma(k-j+1)\Gamma(j+\alpha+1)}\frac{(-t)^{j}}{j!}.
\]

With our normalization, \(\rho\) corresponds to Thangavelu's projective representation \(W(z)=(2\pi)^d\rho(iz)\). Thus the Weyl transform in \cite[p.~11]{T93} is the same integrated Schr\"odinger representation used here for radial functions up to some constant. With this convention, by \cite[(1.3.42)4]{T93}, one has
\begin{equation}\label{Def-laguerre}
\sum_{|\mathbf k|=k}h_{\mathbf k\mathbf k}(z)=L_k^{d-1}\left(\frac12|z|^2\right)\ee^{-|z|^2/4}\eqdef\varphi_k(z).
\end{equation}
We call \(\varphi_k\) the \(k\)-th \emph{Laguerre function}. In this sense, \(\varphi_k\) is the radial phase-space function associated with the \(k\)-th Hermite eigenspace \(\mathcal V_k\).

The operator-theoretic meaning of this correspondence is obtained by integrating the Schr\"odinger representation. Applying the integrated Schrödinger representation to the Laguerre functions gives the fundamental identity
\begin{equation}\label{Ide-projection}
\rho(\varphi_k)=P_k,
\end{equation}
which can be found in \cite[Theorem 1.3.6]{T93}. Therefore the radial phase-space function \(\varphi_k\) corresponds exactly to the
projection onto \(\mathcal V_k\). 
Since \(P_k\) is an orthogonal projection, we have
\[
P_k^*=P_k,\,\,\,P_k^2=P_k.
\]
Hence the spectrum of \(P_k\) is contained in \(\{0,1\}\). On the eigenspace
corresponding to the eigenvalue \(d+2k\), the operator \(P_k\) acts as the identity, while it
vanishes on the orthogonal complement. Therefore the nonzero singular values of \(P_k\) are
all equal to \(1\), and their number is precisely
\[
d_k=\operatorname{Rank}(P_k)
=\#\{\alpha\in\NN^d:|\alpha|=k\}
=\binom{k+d-1}{d-1}.
\]
Consequently, for \(1\le p\leq\infty\), with the convention \(\frac1\infty=0\),
\begin{equation}\label{For-normpPk}
\|P_k\|_{\Scal_p}
=\left(\sum_{j=1}^{d_k}1^p\right)
=d_k^{\frac1p}.
\end{equation}

Identity \eqref{Ide-projection} is the key reason why Laguerre functions are useful in the radial setting. If a radial function \(f\) on \(\RR^{2d}\) admits a Laguerre expansion
\[
    f=\sum_{k=0}^{\infty}c_k\varphi_k,
\]
then
\[
    \rho(f)=\sum_{k=0}^{\infty}c_k\rho(\varphi_k)=\sum_{k=0}^{\infty}c_kP_k.
\]
Thus \(\rho(f)\) is diagonal with respect to the Hermite decomposition, and Schatten norm estimates for \(\rho(f)\) reduce to estimates involving the coefficients \(c_k\) and the projections \(P_k\).

Finally, we record the spectral property of the Laguerre functions that will be
used below. Let
\[
    \mathcal H_{2d}\defeq -\Delta_z+\frac14|z|^2,
    \,\,\, z\in\RR^{2d}
\]
be the \emph{scaled Hermite operator}. Then, by \cite[(1.3.25)]{T93},
\begin{equation}\label{Rela-eigenvalue}
    \mathcal H_{2d}\varphi_k=(2k+d)\varphi_k.
\end{equation}

\begin{remark}
Our notation is partly adapted from M\"uller~\cite[Section~3]{M26} and Thangavelu~\cite[Chapter~1]{T93}. The notation for the Weyl transform, the integrated Schr\"odinger representation, and Schatten classes follows closely that of M\"uller, while the notation for Hermite functions, special Hermite functions, and Laguerre functions follows Thangavelu. Throughout the paper, however, we use the Fourier transform and the integrated Schr\"odinger representation with the normalizations fixed in Section~\ref{Section-pre}; hence some constants involving powers of \(2\pi\) differ from those in the references.

Since \(\varphi_k\) is radial on \(\RR^{2d}\), we shall use the same notation for its radial profile. Thus, for \(z\in\RR^{2d}\) and \(r=|z|\),
\[
    \varphi_k(z)=\varphi_k(r)=L_k^{d-1}\left(\frac12 r^2\right)\ee^{-r^2/4}.
\]

For further background on Hermite functions, special Hermite functions, Laguerre functions, their orthogonality relations, and their asymptotic estimates, we refer to \cite[Chapter~1]{T93}.
\end{remark}

\subsection{Auxiliary estimates for Laguerre functions}\label{Subsection-Laguerreestimates}
We now collect the Laguerre estimates needed in Section~\ref{Section-example}. They concern the symplectic Fourier transform of \(\varphi_k\), the \(L^p\)-asymptotics of \(\varphi_k\), and decay estimates for its derivatives in the region \(|z|\gtrsim \sqrt{k}\).

We first record the \emph{Schwartz-type seminorm} which will be used to measure the regularity and decay of phase-space functions. For \(\phi\in\Scal(\RR^{2d})\) and \(N\in\NN\), define
\begin{equation}\label{Def-N2norm}
\|\phi\|_{(N,2)}
\defeq\sum_{|\alpha|+|\beta|\leq N}\|\partial^\alpha(z^\beta \phi)\|_{L^2(\RR^{2d})}.
\end{equation}

\iffalse
============================Fourier gauss==========================================
\begin{lemma}\label{Lem-guassfourier}
	Let \(a>0\). Then, for \(w\in\RR^{2d}\), 
    \[
    \SF\left(\ee^{-a|z|^2}\right)(w) = (2a)^{-d}\ee^{-\frac{|w|^2}{4a}}. 
    \]
\end{lemma}
\begin{proof}
Set \(f(z)=\ee^{-a|z|^2}\). Since \(f\) is radial and \(J\) is orthogonal, the Euclidean Fourier transform \(\widehat f\) is also radial. Combining with \eqref{Rela-twotransforms}, we deduce that 
\[
\SF(f)(w) =(2\pi)^{-d} \widehat f(Jw)= (2\pi)^{-d} \widehat f(w). 
\]
Using the standard Euclidean Fourier transform of a Gaussian on \(\RR^{2d}\) (see, for instance, Stein--Weiss~\cite[Corollary 1.27 (a)]{SW71}), we have
\[
\widehat f(w)
=\int_{\RR^{2d}}\ee^{-a|z|^2}\ee^{-iw\cdot z}\,\dd z
=\left(\frac{\pi}{a}\right)^d\ee^{-\frac{|w|^2}{4a}}.
\]
Therefore
\[
\SF\left(\ee^{-a|z|^2}\right)(w)=(2\pi)^{-d}\left(\frac{\pi}{a}\right)^d\ee^{-\frac{|w|^2}{4a}}=(2a)^{-d}\ee^{-\frac{|w|^2}{4a}},
\]
as desired.
\end{proof}
================================================================================
\fi
We next compute the symplectic Fourier transform of \(\varphi_k\). This identity will be used in Section~\ref{Section-example} to compare the \(L^p\)-norms of \(\SF\varphi_k\) with those of \(\varphi_k\).
\begin{lemma}\label{Lem-Laguerrefourier} 
For the Laguerre functions \eqref{Def-laguerre} and for \(w\in\RR^{2d}\), we have
\begin{equation}\label{For-Laguerrefourier}
\SF \varphi_k(w) = 2^d(-1)^k L_k^{d-1}\left(2|w|^2\right)\ee^{-|w|^2}=2^d(-1)^k\varphi_k(2w). 
\end{equation}
\end{lemma}
\begin{proof} 
Let \(t\in(-1,1)\). We use the generating function identity for the Laguerre polynomials, 
\begin{equation}\label{For-generating}
\sum_{k=0}^{\infty}L_k^\alpha(x)t^k = (1-t)^{-(\alpha+1)} \ee^{-\frac{t}{1-t}x}, \,\,\,\alpha>-1, 
\end{equation}
which can be found in \cite[(1.1.45)]{T93}. Taking \(\alpha=d-1\) and \(x=\frac12 r^2\) in \eqref{For-generating} with \(r=|z|\), gives 
\begin{equation}\label{For-Gt}
G_t(r) \defeq \sum_{k=0}^{\infty} L_k^{d-1}\left(\frac12 r^2\right)t^k \ee^{-r^2/4} = (1-t)^{-d} \ee^{-\frac{1+t}{4(1-t)}r^2}. 
\end{equation}
Since \(\frac{1+t}{4(1-t)}>0\), using the standard Euclidean Fourier transform of a Gaussian on \(\RR^{2d}\), see \cite[Corollary~1.27(a)]{SW71}, together with \eqref{Rela-twotransforms}, we have
\[
\begin{aligned} 
\SF(G_t)(w) %= (1-t)^{-d} \left(2\cdot \frac{1+t}{4(1-t)}\right)^{-d}  \ee^{-\frac{|w|^2}{4\cdot \frac{1+t}{4(1-t)}}} 
= 2^d(1+t)^{-d} \ee^{-|w|^2\frac{1-t}{1+t}}. 
\end{aligned} 
\] 
On the other hand, combining with the definition of \(G_t\) in \eqref{For-Gt}, 
\[
\SF(G_t)(w)=2^dG_{-t}(2s),\,\,\,s=|w|.
\] 
Comparing the coefficients of \(t^k\) on both sides, we obtain 
\[ 
\SF \varphi_k(w)=\SF\left[ L_k^{d-1}\left(\frac12 |z|^2\right)\ee^{-|z|^2/4} \right](w) = 2^d(-1)^k L_k^{d-1}\left(2|w|^2\right)\ee^{-|w|^2}=2^d(-1)^k\varphi_k(2w).
\] 
\end{proof}

We next recall the \(L^p\)-asymptotics of the Laguerre functions. These estimates will be used in Section~\ref{Section-example} to determine the leading-order size of the examples constructed from \(\varphi_k\). We omit the critical case \(p=p_c\), where an additional logarithmic factor appears, since it will not be used below.
\begin{lemma}\label{Lem-Lpnormphik}
Set
\(
p_c\defeq \frac{4d}{2d-1}.
\)
Then, as \(k\to\infty\), the following asymptotics hold, with implicit constants independent of \(k\):
\begin{enumerate}
\item If \(1\le p<p_c\), then
\[
\|\varphi_k\|_{L^p(\RR^{2d})}
\simeq_{d,p}
k^{-\frac12+\frac{d}{p}}.
\]

\item If \(p_c<p\le\infty\), then
\[
\|\varphi_k\|_{L^p(\RR^{2d})}
\simeq_{d,p}
k^{d-1-\frac{d}{p}}.
\]
\end{enumerate}
\end{lemma}
\begin{proof}
We use the definition of Thangavelu's normalized Laguerre functions in \cite[(1.4.10)]{T93},
\begin{equation}\label{Def-mathcalL}
\mathcal{L}_k^\alpha\left(\frac{r^2}{2}\right)
=\left(\frac{k!}{(k+\alpha)!}\right)^{1/2}\ee^{-r^2/4}\left(\frac{r^2}{2}\right)^{\alpha/2}L_k^\alpha\left(\frac{r^2}{2}\right).
\end{equation}
For \(\alpha=d-1\), this gives the relation
\begin{equation}\label{Rela-LL}
\varphi_k(r)
=\left(\frac{k!}{(k+d-1)!}\right)^{-1/2}\left(\frac{r^2}{2}\right)^{-\frac{d-1}{2}}\mathcal L_k^{d-1}\left(\frac{r^2}{2}\right).
\end{equation}
Using polar coordinates and the change of variables \(t=\frac{1}{2}r^2\), %so that
%\[
%r^{2d-1}\,\dd r=(2t)^{d-1}\,\dd t,
%\]
we obtain, for \(1\le p<\infty\),
\begin{equation}\label{Est-varphiLpnorm}
\begin{aligned}
\|\varphi_k\|_{L^p(\RR^{2d})}^p
&\simeq_d \int_0^\infty |\varphi_k(r)|^p r^{2d-1}\,\dd r  \\
&\simeq_d \left(\frac{k!}{(k+d-1)!}\right)^{-p/2}\int_0^\infty\left|\mathcal L_k^{d-1}(t)\right|^p t^{(d-1)(1-\frac{p}{2})}\,\dd t  \\
&\simeq_dk^{\frac{p(d-1)}{2}}\left\|\mathcal L_k^{d-1}(t)\,t^{(d-1)(\frac1p-\frac12)}\right\|_{L^p(\RR_+)}^p.
\end{aligned}
\end{equation}
We apply here Markett's weighted \(L^p\)-estimate for normalized Laguerre functions. Specialized to the present parameters, \cite[Lemma~1]{M82} gives the following three regimes as \(k\to\infty\):
\begin{itemize}
		\item If $1\le p\le 4$ and $\beta<\tfrac{2}{p}-\tfrac12$, then
		$$
		\big\| \mathcal L_k^{\alpha+\beta}(x)\,x^{-\beta/2}\big\|_{L^p(\RR_{+})}\simeq k^{\frac{1}{p}-\frac12-\frac{\beta}{2}}.
		$$
		\item If $1\le p\le 4$ and $\beta>\tfrac{2}{p}-\tfrac12$, then
		$$
		\big\| \mathcal L_k^{\alpha+\beta}(x)\,x^{-\beta/2}\big\|_{L^p(\RR^{+})}\simeq k^{\frac{\beta}{2}-\frac{1}{p}}.
		$$
		\item If $4<p\leq\infty$ and $\beta>\tfrac{4}{3p}-\tfrac13$, then
		$$
		\big\| \mathcal L_k^{\alpha+\beta}(x)\,x^{-\beta/2}\big\|_{L^p(\RR^{+})}\simeq k^{\frac{\beta}{2}-\frac{1}{p}}.
		$$
	\end{itemize}
    Put 
    \[ 
    \beta\defeq 2(d-1)\left(\frac12-\frac1p\right), \,\,\,\alpha=\frac{2}{p}(d-1)>-\frac2p,\,\,\,\alpha+\beta=d-1>-1. 
    \] 
     Therefore the weighted norm appearing above can be written as 
     \[ 
     \left\| \mathcal L_k^{d-1}(t)\, t^{(d-1)(\frac1p-\frac12)} \right\|_{L^p(\RR_+)} = \left\| \mathcal L_k^{\alpha+\beta}(t)\, t^{-\beta/2} \right\|_{L^p(\RR_+)}. 
     \]
     Substituting these estimates into \eqref{Est-varphiLpnorm} and simplifying the powers of \(k\), we obtain the two stated asymptotic regimes. Indeed, the condition \(\beta=2/p-1/2\) is equivalent to \(p=p_c\), which is the reason why we have this critical point.
     \end{proof}

Finally, we need a decay estimate for derivatives of \(\varphi_k\). This will allow us, in Section~\ref{Section-example}, to cut off \(\varphi_k\) outside a ball of radius comparable to \(\sqrt{k}\) while producing only an exponentially small error.

\begin{lemma}\label{Lem-Derivativephik}
Let \(\alpha\in\NN^{2d}\) be a multi-index and set \(m=|\alpha|\). 
There exists a constant \(\gamma>0\) such that, whenever
\[
|z|\ge \sqrt{6(2k+d+m)},
\]
we have
\[
|\partial^\alpha\varphi_k(z)|\lesssim_{d,\alpha}\ee^{-\frac{\gamma}{4}|z|^2}.
\]
\end{lemma}

\begin{proof}
We use the notation 
\[
T_k^b(t)\defeq (-1)^k\ee^{-t}L_k^b(2t),
\]
which corresponds to the notation \(\mathcal L_k^{(b)}\) used in \cite{M15}. By the definitions of \(\varphi_k\) in \eqref{Def-laguerre} and \(\mathcal L_{k}^b\) in \eqref{Def-mathcalL}, we have 
\begin{equation}\label{Rela-varphiT}
\varphi_k(z)=(-1)^k T_k^{d-1}\left(\frac{|z|^2}{4}\right),
\end{equation}
and
\begin{equation}\label{Rela-LT}
		|T_{k}^{b}(t)|=\,2^{-\frac{b}{2}}\left(\frac{k!}{(k+b)!}\right)^{-\frac{1}{2}} \,t^{-\frac{b}{2}}\,|\mathcal{L}_{k}^{b}(2t)|.
	\end{equation}
By \cite[Lemma~5]{M15}, one has
\begin{equation*}
\partial_t T_k^b(t)=T_{k-1}^{b+1}(t)-T_k^{b+1}(t).
\end{equation*}
Iterating this identity gives, for every integer \(\ell\ge0\),
\begin{equation}\label{For-chain}
\partial_t^\ell T_k^b(t)
=\sum_{j=0}^{\ell}(-1)^{\ell-j}\binom{\ell}{j}T_{k-j}^{b+\ell}(t),
\end{equation}
where \(T_{k-j}^{b+\ell}\equiv0\) if \(k-j<0\).

Set \(t=|z|^2/4\). By repeated use of the chain rule starting from \eqref{Rela-varphiT}, \(\partial^\alpha\varphi_k(z)\) is a finite sum of terms of the form
\[
P_{\alpha,\ell}(z)\,
\partial_t^\ell T_k^{d-1}(t),
\,\,\, 0\le \ell\le m,
\]
where \(P_{\alpha,\ell}\) is a polynomial of degree at most \(m\). Hence by \eqref{For-chain},
\begin{equation}\label{Est-derivativevarphik}
|\partial^\alpha\varphi_k(z)|
\lesssim_\alpha(1+|z|)^m\sum_{\ell=0}^{m}\sum_{j=0}^{\ell}\left|T_{k-j}^{d-1+\ell}\left(\frac{|z|^2}{4}\right)\right|.
\end{equation}
	Now we want to reduce the estimate for \(T_h^b\) to the estimate of \(\mathcal{L}_h^b\), which can be found in \cite[Lemma 1.5.3]{T93}; according to this estimate, in the case where \(x\geq 6h+3(b+1)\), one obtains 
	\[
	|\mathcal{L}_{h}^{b}(x)|\lesssim\ee^{-\gamma x},\,\,\mathrm{for}\,\,\mathrm{fixed}\,\,\gamma>0.
	\]
Once we have this estimate, together with the relation \eqref{Rela-LT} between
\(\mathcal L_h^b\) and \(T_h^b\), it follows that
\[
\left|T_h^b\left(\frac{|z|^2}{4}\right)\right|\lesssim_b(1+h)^{\frac b2}|z|^{-b}\ee^{-\frac{\gamma}{2}|z|^2},
\,\,\,|z|\ge \sqrt{6(2h+b+1)}.
\]
Now assume that \(|z|\ge \sqrt{6(2k+d+m)}\), since
\[
(2(k-j)+d-1+\ell+1)\leq 2(k-j)+d+\ell\leq 2k+d+m
\]
for all \(0\le j\le \ell\le m\), then
\[
\left|
T_{k-j}^{d-1+\ell}\left(\frac{|z|^2}{4}\right)
\right|\lesssim_{d,\alpha}(1+k-j)^{\frac{d-1+\ell}{2}}|z|^{-(d-1+\ell)}\ee^{-\frac{\gamma}{2}|z|^2}, \,\,\forall \,0\le j\le \ell\le m.
\]
Combining this with \eqref{Est-derivativevarphik}, we obtain
\[
\begin{aligned}
|\partial^\alpha\varphi_k(z)|
\lesssim_{d,\alpha}(1+|z|)^m\sum_{\ell=0}^{m}\sum_{j=0}^{\ell}(1+k-j)^{\frac{d-1+\ell}{2}}|z|^{-(d-1+\ell)}\ee^{-\frac{\gamma}{2}|z|^2}.
\end{aligned}
\]
Since \(|z|^2\ge 6(2k+d+m)\), we have \(1+k-j\lesssim |z|^2\). Hence
\[
(1+k-j)^{\frac{d-1+\ell}{2}}|z|^{-(d-1+\ell)}\lesssim_{d,\alpha} 1.
\]
Therefore
\[
|\partial^\alpha\varphi_k(z)|\lesssim_{d,\alpha}(1+|z|)^m\ee^{-\frac{\gamma}{2}|z|^2}.
\]
Finally, after decreasing \(\gamma\) if necessary, the polynomial factor can be absorbed into the exponential decay, and we get the desired estimate.
\end{proof}
The estimates collected in this subsection will be used in Section~\ref{Section-example} to analyze the truncated Laguerre examples. 

\section{Refined Schatten restriction estimates}\label{Section-Mainestimate}
In this section we prove the main comparison estimates with explicit dependence on the support radius. The argument follows the general strategy of M\"uller \cite{M26}: one first proves endpoint estimates for regularized operators, identifies the adjoint estimate through Werner convolution and Schatten duality, and then interpolates between the endpoints. The main improvement occurs at the trace class endpoint. In \cite[Lemma 3.2]{M26}, the corresponding endpoint estimate is obtained from the full Hermite expansion and leads to a loss controlled by a high-order Schwartz seminorm. Here we use the radial Hermite--Laguerre correspondence from Section~\ref{Section-Laguerre} to obtain a sharper trace-class estimate for radial symbols; see Lemma~\ref{Lem-S1normN2norm}. This refinement is the source of the improved power \(2d+1+\varepsilon\) for any \(\varepsilon>0\) in Theorem~\ref{Mainthm}.

We first recall the Hilbert--Schmidt estimate, also known as the Plancherel theorem for the Weyl transform. This result is standard, and we omit the proof; see, for instance, \cite[Theorem~1.2.1]{T93}.

\begin{theorem}\label{Thm-schmidthilbert}
%\textcolor{red}{Consider the linear map
%\begin{equation}\label{Def-operatorT}
    %T:L^2(\RR^{2d})\to\Scal_2, \,\,\, T(g)=L_g.
%\end{equation}
%Then,
\begin{equation*}%\label{Idn-schmidthilbert}
   % \|Tg\|_{\Scal_2}=
   \|L_g\|_{\Scal_2}=(2\pi)^{-d/2}\|g\|_{L^2(\RR^{2d})}.
\end{equation*}
%Moreover, \(T\) is an isometric isomorphism up to the normalization constant \((2\pi)^{-d/2}\).}
\end{theorem}

The preceding theorem gives the exact \(\Scal_2\)-endpoint. We next record a simple technical estimate which will be used to pass from Hermite spectral information to the Schwartz-type seminorms \(\|\cdot\|_{(t,2)}\) defined in \eqref{Def-N2norm}.

\begin{lemma}\label{Lem-revised}
Let \(t\in\NN\). Then, for every \(f\in\Scal(\RR^{2d})\), 
\[
\left\|(1+\mathcal H_{2d})^{\frac t2}f\right\|_{L^2(\RR^{2d})}
\lesssim_{d,t}
\|f\|_{(t,2)}.
\]
\end{lemma}

\begin{proof}
Recall that
\[
\mathcal H_{2d}=-\Delta_z+\frac14|z|^2,\,\,\,z\in \RR^{2d}.
\]
For \(j=1,\dots,2d\), define
\[
A_j\defeq-\partial_{z_j}+\frac{z_j}{2},
\,\,\,
A_j^*\defeq\partial_{z_j}+\frac{z_j}{2}.
\]
Then \(A_j^*\) is the \(L^2(\RR^{2d})\)-adjoint of \(A_j\). A direct computation gives
\[
A_jA_j^*
=\left(-\partial_{z_j}+\frac{z_j}{2}\right)\left(\partial_{z_j}+\frac{z_j}{2}\right)
=-\partial_{z_j}^2+\frac14 z_j^2-\frac12.
\]
Hence
\begin{equation}\label{Expre-H2d}
\mathcal H_{2d}=\sum_{j=1}^{2d}A_jA_j^*+d,
\end{equation}
which is the analogue of \cite[(1.1.29)]{T93} for the scaled Hermite operator. Thus \(\mathcal H_{2d}\) can be controlled through the operators \(A_j\) and \(A_j^*\), while each \(A_j\) is simply a first-order operator of the form ``multiplication by \(z_j\) minus derivative.''

We also record the commutation relation. Since
\[
\left[\partial_{z_i},z_j\right]=\delta_{ij},
\]
we have, for any \(i,j=1,\cdots,2d\),
\begin{equation}\label{For-commutatorAAstar}
[A_i^*,A_j]=\delta_{ij},
\,\,\,
[A_i,A_j]=[A_i^*,A_j^*]=0.
\end{equation}

Let \(\{\psi_\nu\}_{\nu\in\NN^{2d}}\) denote the orthonormal basis of
\(L^2(\RR^{2d})\) such that
\begin{equation}\label{scaledHermitebasis}
\mathcal H_{2d}\psi_\nu=(|\nu|+d)\psi_\nu,\,\,\,\nu\in\NN^{2d}.
\end{equation}
Precisely, if \(\{h_\nu\}_{\nu\in \NN^{2d}}\) is the standard Hermite basis, then from \eqref{Hermitebasis}, 
\[
H_{2d}h_\nu=(-\Delta_z+|z|^2)h_\nu=2(d+|\nu|)h_\nu.
\]
Define 
\[
\psi_\nu\defeq 2^{-d/2}h_\nu\left(\frac{z}{\sqrt2}\right),
\]
then it satisfies \eqref{scaledHermitebasis}. We now explain the raising relation. First, by \eqref{Expre-H2d} and \eqref{scaledHermitebasis},
\[
\sum_{j=1}^{2d}A_jA_j^*\psi_\nu=|\nu|\psi_\nu.
\]

\iffalse
==============================raising operator====================================
That is, \(A_jA_j^*\) counts the \(j\)-th Hermite level. The operator \(A_j\) raises the \(j\)-th Hermite index, i.e.,
\[
    A_j\psi_\nu=c_{\nu,j}\psi_{\nu+e_j}
\]
for some constant \(c_{\nu,j}\). To determine \(c_{\nu,j}\), we use
\[
    A_j^*A_j=A_jA_j^*+1,
\]
which follows from \eqref{For-commutatorAAstar}. Then
\[
\|A_j\psi_\nu\|_2^2
=\langle A_j^*A_j\psi_\nu,\psi_\nu\rangle
=\langle (A_jA_j^*+1)\psi_\nu,\psi_\nu\rangle=\langle (\nu_j+1)\psi_\nu,\psi_\nu\rangle=\nu_j+1.
\]
With the standard real normalization of the Hermite basis, the coefficient is
positive. We get
===============================================================================
\fi

On the other hand, with our normalization, after the corresponding scaling, the analogue of the standard raising relation \cite[(1.1.30)]{T93} is 
\begin{equation}\label{For-raisingrelation}
A_j\psi_\nu=\sqrt{\nu_j+1}\,\psi_{\nu+e_j}.
\end{equation}

Now write
\begin{equation}\label{For-fexpansion}
f=\sum_{\nu\in\NN^{2d}}c_\nu\psi_\nu.
\end{equation}
By the spectral theorem and combining with \eqref{scaledHermitebasis},
\begin{equation}\label{spectralHermiteNorm}
\left\|(1+\mathcal H_{2d})^{\frac t2}f\right\|_2^2
=\sum_{\nu\in\NN^{2d}}(1+|\nu|+d)^t |c_\nu|^2.
\end{equation}

We next compare this with the \(L^2\)-norms of products of the creation operators.
For a multi-index \(\alpha\in\NN^{2d}\), set \(A^\alpha \defeq A_1^{\alpha_1}\cdots A_{2d}^{\alpha_{2d}}.\) Using \eqref{For-raisingrelation} repeatedly gives
\[
A^\alpha\psi_\nu=C_{\nu,\alpha}\psi_{\nu+\alpha},
\]
where
\[
|C_{\nu,\alpha}|^2=\prod_{j=1}^{2d}(\nu_j+1)(\nu_j+2)\cdots(\nu_j+\alpha_j) \simeq_{d,\alpha}\prod_{j=1}^{2d}(\nu_j+1)^{\alpha_j}
\]
since
\[
%\prod_{j=1}^{2d}
(\nu_j+1)^{\alpha_j}\le%\prod_{j=1}^{2d}
(\nu_j+1)(\nu_j+2)\cdots(\nu_j+\alpha_j) \le%\prod_{j=1}^{2d}
\alpha_j^{\alpha_j}(\nu_j+1)^{\alpha_j},\,\,\,\forall \,j=1,\cdots,2d.
\]
Since the family \(\{\psi_\nu\}\) is orthonormal, together with \eqref{For-fexpansion}, we obtain
\[
\|A^\alpha f\|_2^2=\sum_{\nu\in\NN^{2d}}|C_{\nu,\alpha}|^2 |c_\nu|^2.
\]
Therefore
\begin{equation}\label{Est-Aalpha}
\sum_{|\alpha|\le t}\|A^\alpha f\|_2^2\simeq_{d,t}\sum_{\nu\in\NN^{2d}}\left(\sum_{|\alpha|\le t}\prod_{j=1}^{2d}(\nu_j+1)^{\alpha_j}\right)|c_\nu|^2.
\end{equation}
Using the elementary inequality
\[
(1+|\nu|+d)^t\simeq_{d,t}\sum_{|\alpha|\le t}\prod_{j=1}^{2d}(\nu_j+1)^{\alpha_j},
\]
we get, together with \eqref{spectralHermiteNorm} and \eqref{Est-Aalpha},
\begin{equation}\label{Esti-DominationmathcalL}
\left\|(1+\mathcal H_{2d})^{\frac t2}f\right\|_2
\simeq_{d,t}
\sum_{|\alpha|\le t}\|A^\alpha f\|_2.
\end{equation}

It remains to compare the right-hand side with \(\|f\|_{(t,2)}\). Because of the definition of \(A_j\), each product \(A^\alpha\), with \(|\alpha|\le t\), is a finite linear combination of operators of the form
\[
z^\beta\partial^\gamma,\,\,\,|\beta|+|\gamma|\le t.
\]
Hence by \eqref{Def-N2norm},
\begin{equation}\label{Est-A}
\|A^\alpha f\|_2\lesssim_{d,t}\sum_{|\beta|+|\gamma|\le t}\|z^\beta\partial^\gamma f\|_2.
\end{equation}
On the other hand, we claim
\[
    \sum_{|\beta|+|\gamma|\le t}\|z^\beta\partial^\gamma f\|_2
    \simeq_{d,t}
    \sum_{|\beta|+|\gamma|\le t}\|\partial^\gamma(z^\beta f)\|_2 .
\]
By Leibniz' formula,
\[
    \partial^\gamma(z^\beta f)=z^\beta\partial^\gamma f+ \sum_{\substack{0<\nu\le\gamma\\ \nu\le\beta}} c_{\beta,\gamma,\nu}\,z^{\beta-\nu}\partial^{\gamma-\nu}f .
\]
This immediately gives one inequality between the two families of seminorms.
The reverse inequality follows by induction on \(|\beta|+|\gamma|\), since all
terms in the sum have strictly smaller total order. Hence we obtain the claim and then combining with \eqref{Est-A},
\[
\|A^\alpha f\|_2\lesssim_{d,t}\sum_{|\beta|+|\gamma|\le t}\|\partial^\gamma(z^\beta f)\|_2
=\|f\|_{(t,2)}.
\]
Combining this with \eqref{Esti-DominationmathcalL}, we conclude the result.
%\[
%\left\|(1+\mathcal H_{2d})^{\frac t2}f\right\|_{L^2(\RR^{2d})}\lesssim_{d,t}\|f\|_{(t,2)},
%\]
%as desired.
\end{proof}

We now prove the trace-class endpoint estimate. This is the main refinement of \cite[Lemma 3.2]{M26} in the radial setting: the Hermite--Laguerre correspondence allows us to replace the full Hermite expansion by a diagonal Laguerre expansion and obtain a lower-order seminorm.
\begin{lemma}\label{Lem-S1normN2norm}
Let \(f\in \Scal(\RR^{2d})\) be radial. Then
\begin{equation}\label{Ine-S1normlemma}
    \|L_{\SF f}\|_{\Scal_1}\lesssim_{d,s}\|(1+\mathcal H_{2d})^{\frac s2}f\|_{L^2(\RR^{2d})},\,\,\,\forall s>d+1.
\end{equation}
In particular, taking \(s=d+2\), we have
\begin{equation}\label{Ine2-S1normlemma}
\|L_{\SF f}\|_{\Scal_1}\lesssim_d\|f\|_{(d+2,2)}.
\end{equation}
\end{lemma}

\begin{proof}
Let \(f\in\Scal(\RR^{2d})\) be radial. Then \(f\) admits the Laguerre expansion
\[
f=\sum_{k=0}^{\infty}c_k\varphi_k,
\,\,\,
c_k=\frac{(f,\varphi_k)_{L^2(\RR^{2d})}}{\|\varphi_k\|_{L^2(\RR^{2d})}^2}.
\]
Since \(\rho(\varphi_k)=P_k\), we have
\[
L_{\SF f}
=\rho(f)
=\sum_{k=0}^{\infty}c_kP_k.
\]
The projections \(P_k\) are mutually orthogonal. Hence
\[
\|L_{\SF f}\|_{\Scal_1}
=\sum_{k=0}^{\infty}|c_k|d_k.
\]

On the other hand, by \eqref{Ide-projection}, the Hilbert--Schmidt identity in Theorem~\ref{Thm-schmidthilbert}, and Lemma~\ref{Lem-Laguerrefourier}, we have
\[
\begin{aligned}
    \|P_k\|_{\Scal_2}
    &=\|\rho(\varphi_k)\|_{\Scal_2}
     =\|L_{\SF\varphi_k}\|_{\Scal_2} 
     =(2\pi)^{-d/2}\|\SF\varphi_k\|_{L^2(\RR^{2d})}  \\
    &=(2\pi)^{-d/2}2^d\|\varphi_k(2\cdot)\|_{L^2(\RR^{2d})}
    =(2\pi)^{-d/2}\|\varphi_k\|_{L^2(\RR^{2d})}.
\end{aligned}
\]
Combining this with \eqref{For-normpPk} for \(p=2\), namely
\(\|P_k\|_{\Scal_2}=d_k^{1/2}\), gives
\begin{equation}\label{Est-varphiL2norm}
    \|\varphi_k\|_{L^2(\RR^{2d})}^2=(2\pi)^d d_k.
\end{equation}
Consequently,
\[
    \|L_{\SF f}\|_{\Scal_1}=\sum_{k=0}^{\infty}|c_k|d_k
    =\sum_{k=0}^{\infty}\frac{|(f,\varphi_k)_{L^2(\RR^{2d})}|}{\|\varphi_k\|_2^2}d_k
    =(2\pi)^{-d}\sum_{k=0}^{\infty}|(f,\varphi_k)_{L^2(\RR^{2d})}|.
\]

We now estimate \((f,\varphi_k)_{L^2(\RR^{2d})}\). %Let \(s>d+1\). Because of \eqref{Rela-eigenvalue} and \eqref{Est-varphiL2norm}, we have
%\[
%\left\|(1+\mathcal H_{2d})^{-\frac s2}\varphi_k\right\|_{L^2(\RR^{2d})}
%=(1+2k+d)^{-\frac s2}\|\varphi_k\|_{L^2(\RR^{2d})}\simeq d_{k}^{1/2}(1+k)^{-\frac s2}.
%\]
The spectral theorem, combining with \eqref{Rela-eigenvalue}, gives, for every real \(s\ge0\),
\[
    (f,\varphi_k)_{L^2(\RR^{2d})}=\left( (1+\mathcal H_{2d})^{s/2}f,(1+\mathcal H_{2d})^{-s/2}\varphi_k\right)=(1+2k+d)^{-s/2}\left( (1+\mathcal H_{2d})^{s/2}f,\varphi_k\right).
\]
Hence, by Cauchy's inequality and \eqref{Est-varphiL2norm},
\[
\begin{aligned}
    |(f,\varphi_k)_{L^2(\RR^{2d})}|
    &\le(1+2k+d)^{-s/2}\left\|(1+\mathcal H_{2d})^{s/2}f\right\|_{L^2(\RR^{2d})}\|\varphi_k\|_{L^2(\RR^{2d})}  \\
    &\lesssim_{d,s}\|(1+\mathcal H_{2d})^{s/2}f\|_{L^2(\RR^{2d})}(1+k)^{-s/2}d_k^{1/2} .
\end{aligned}
\]
Therefore, using \(d_k\simeq_d(1+k)^{d-1}\),
\[
\begin{aligned}
    \left\|L_{\SF f}\right\|_{\Scal_1}
    %&\lesssim_{d,s}\sum_{k=0}^{\infty}(1+k)^{-s/2}d_k^{1/2}\left\|(1+\mathcal H_{2d})^{s/2}f\right\|_{L^2(\RR^{2d})} \\
    &\lesssim_{d,s}\left\|(1+\mathcal H_{2d})^{s/2}f\right\|_{L^2(\RR^{2d})}\sum_{k=0}^{\infty}(1+k)^{-\frac{s-d+1}{2}}.
\end{aligned}
\]
The series converges precisely when \(s>d+1\), which gives \eqref{Ine-S1normlemma}. Moreover, by Lemma~\ref{Lem-revised},
\begin{equation}\label{Est-scaledN2noem}
\left\|(1+\mathcal H_{2d})^{\frac s2}f\right\|_{L^2(\RR^{2d})}
\lesssim_{d,s}\|f\|_{(s,2)}.
\end{equation}
Setting \(s=d+2\), we finally obtain \eqref{Ine2-S1normlemma}.
\end{proof}

In what follows, choose a \textbf{radial} function \(\chi\in C_0^\infty(\RR^{2d})\) such that 
\[ 
\chi\equiv 1 \,\, \mathrm{on }\,\,\overline{B(0,1)}, \,\,\, \chi\equiv 0 \,\, \mathrm{on } \,\,B(0,2)^c, 
\] 
and set 
\[ 
\chi_R(z)\defeq \chi(z/R). 
\]

We now apply Lemma~\ref{Lem-S1normN2norm} to the radial cutoff functions \(\chi_R\). Since the derivatives of \(\chi_R(z)=\chi(z/R)\) scale explicitly with \(R\), one can easily get \(\left\|(1+\mathcal H_{2d})^{\frac s2}\chi_R\right\|_{L^2(\RR^{2d})}\lesssim R^{d+s}\) for \(R\geq 1\). This is the point at which the improvement over the \(R^{5d+2}\)-loss in \cite{M26} enters the proof.

\begin{theorem}\label{Thm-oneway}
Let \(R\ge1\), \(\psi_R\defeq \SF(\chi_R)\), and let \(v\in\Scal'(\RR^{2d})\) be such that \(L_v\in\Scal_p\). Then, for every
\(1\le p\le\infty\) and every \(\varepsilon>0\),
\begin{equation}\label{MainIneq-oneway}
\|v*\psi_R\|_{L^p(\RR^{2d})}\lesssim_{d,p,\varepsilon} R^{2d+1+\varepsilon}\|L_v\|_{\Scal_p}.
\end{equation}
\end{theorem}

\begin{proof}
Since \(\psi_R=\SF(\chi_R)\) and \(\SF^2=\Id\), Lemma~\ref{Lem-WernerConvolution} together with \eqref{Conv-sym} gives, in \(\Scal'(\RR^{2d})\),
\[
    \rho(\SF v)\star\rho(\chi_R)=\SF\big((\SF v)\chi_R\big) =(2\pi)^{-d}v*\psi_R .
\]
Applying Young's inequality for Werner convolution in Proposition~\ref{Prop-Younginequality} and \eqref{Ine-S1normlemma}, for \(s>d+1\), we obtain
\[
\begin{aligned}
&\|v*\psi_R\|_{L^p(\RR^{2d})}
\simeq_d\|\rho(\SF v)\star\rho(\chi_R)\|_{L^p(\RR^{2d})}\\
&\le\|\rho(\SF v)\|_{\Scal_p}\|\rho(\chi_R)\|_{\Scal_1}\lesssim_{d,p,s} \|L_v\|_{\Scal_p}\left\|(1+\mathcal H_{2d})^{\frac s2}\chi_R\right\|_{L^2(\RR^{2d})}.
\end{aligned}
\]
It remains to estimate \(\left\|(1+\mathcal H_{2d})^{s/2}\chi_R\right\|_{L^2(\RR^{2d})}\). By the scaling argument in \cite[Proof of Theorem~4.2]{M26} and \eqref{Est-scaledN2noem}, for every
integer \(N\ge0\),
\[
    \left\|(1+\mathcal H_{2d})^{N/2}\chi_R\right\|_{L^2(\RR^{2d})}\lesssim_{d,N}\|\chi_R\|_{N,2}\lesssim_{d,N}R^{d+N},\,\,\, R\ge1.
\]
Let \(s\in \RR\) and \(s>d+1\), and choose integers \(N_0\le s\le N_1\) with \(s=(1-\theta)N_0+\theta N_1\). Since \(1+\mathcal H_{2d}\) is positive self-adjoint, the spectral theorem gives
\[
    \left\|(1+\mathcal H_{2d})^{s/2}\chi_R\right\|_2
    \le\left\|(1+\mathcal H_{2d})^{N_0/2}\chi_R\right\|_2^{1-\theta}\left\|(1+\mathcal H_{2d})^{N_1/2}\chi_R\right\|_2^\theta .
\]
Hence
\[
    \left\|(1+\mathcal H_{2d})^{s/2}\chi_R\right\|_{L^2(\RR^{2d})}\lesssim_{d,s}R^{d+s},\,\,\, R\ge1.
\]
Combining this with the preceding estimates gives the required result.
%\[
   % \|v*\psi_R\|_{L^p(\RR^{2d})}\lesssim  R^{2d+2}\|L_v\|_{\Scal_p}.
%\]
%For\(|\alpha|+|\beta|\le d+2\), the Leibniz rule and the scaling \(\chi_R(z)=\chi(z/R)\) give
%\[
%\|\partial^\alpha(z^\beta\chi_R)\|_{L^2(\RR^{2d})}\lesssim_{\alpha,\beta,\chi}R^{|\beta|-|\alpha|+d}.
%\]
%Indeed, the factor \(R^d\) comes from the \(L^2\)-scaling in the ambient space \(\RR^{2d}\). Since \(|\alpha|+|\beta|\le d+2\), we have
%\[
%|\beta|-|\alpha|+d\le 2d+2.
%\]
%Thus, for \(R\ge1\),
%\[
%\|\chi_R\|_{(d+2,2)}\lesssim R^{2d+2}.
%\]
%Combining the above estimates gives
%\[
%\|v*\psi_R\|_{L^p(\RR^{2d})} \lesssim R^{2d+2}\|L_v\|_{\Scal_p}.
%\]
%This proves the theorem.
\end{proof}

The preceding theorem proves one direction of the regularized comparison: it controls \(v*\psi_R\) in \(L^p(\RR^{2d})\) by the Schatten norm of \(L_v\). We now prove the converse estimate, namely a bound for \(L_{g*\psi_R}\) in \(\Scal_p\) in terms of \(\|g\|_{L^p(\RR^{2d})}\). We derive it by duality, and the same duality mechanism will also be used later in Corollary \ref{Cor-gammap}.

\begin{theorem}\label{Thm-dualityPsi}
Let \(R\ge1\), and let \(\psi_R\) be as in
Theorem~\ref{Thm-oneway}. Then, for every \(1\le p\leq\infty\) and every \(\varepsilon>0\),
\begin{equation}\label{Norm-dualityPsi}
(2\pi)^{-d}\left\|A\mapsto (L^{-1}A)*\psi_R\right\|_{\Scal_p\to L^p(\RR^{2d})}
=\left\| g\mapsto L_{g*\psi_R}\right\|_{L^{p'}(\RR^{2d})\to\Scal_{p'}}\lesssim_{d,p,\varepsilon} R^{2d+1+\varepsilon}.
\end{equation}
Here \(L^{-1}\) is understood in the distributional sense.
\end{theorem}
\begin{proof}
By Remark~\ref{Rem-Linjective}, the Weyl correspondence \(f\mapsto L_f\) is a topological isomorphism at the level of tempered distributions. Hence \(L^{-1}A\) is well defined distributionally for every \(A\in\Scal_p\). Define
\begin{equation}\label{Def-PR}
    \PR A\defeq(2\pi)^{-d}(L^{-1}A)*\psi_R,\,\,\, A\in \Scal_p.
\end{equation}
Set
\[
    v\defeq (2\pi)^{-d}L^{-1}A
\]
in the distributional sense. Then
\[
    L_v=(2\pi)^{-d}A,\,\,\,
    \PR A=v*\psi_R.
\]
Therefore, by Theorem~\ref{Thm-oneway},
\begin{equation}\label{Norm-PR}
\begin{aligned}
    \|\PR A\|_{L^p(\RR^{2d})}=
    \|v*\psi_R\|_{L^p(\RR^{2d})}\lesssim_{d,p,\varepsilon}
    R^{2d+1+\varepsilon}\|L_v\|_{\Scal_p}\simeq_d
    R^{2d+1+\varepsilon}\|A\|_{\Scal_p}.
\end{aligned}
\end{equation}
Next we want to identify the Banach adjoint of \(\PR\). We first compute the adjoint on the Hilbert--Schmidt level. Let \(A\in\Scal_2\), and write \(A=L_f\), where \(f=L^{-1}A\in L^2(\RR^{2d})\). Then, for \(h\in L^2(\RR^{2d})\), the Hilbert--Schmidt identity gives
\[
    (L_h,A)_{\Scal_2}
    =
    (L_h,L_f)_{\Scal_2}
    =
    (2\pi)^{-d}(h,f)_{L^2(\RR^{2d})}.
\]
Thus the Hilbert-space adjoint of \(L\) is given by
\[
    L^*A=(2\pi)^{-d}L^{-1}A .
\]
Consequently,
\[
    \PR A=(L^*A)*\psi_R .
\]

Since \(\chi_R\) is real-valued and even, we have
\[
    \psi_R^*(z)\defeq \overline{\psi_R(-z)}=\psi_R(z).
\]
Therefore, for \(A\in\Scal_2\) and \(g\in L^2(\RR^{2d})\),
\[
\begin{aligned}
    (\PR A,g)_{L^2}=((L^*A)*\psi_R,g)_{L^2} = (L^*A,g*\psi_R^*)_{L^2}=(L^*A,g*\psi_R)_{L^2}= (A,L_{g*\psi_R})_{\Scal_2}.
\end{aligned}
\]
Hence,
\[
    \PR^*g=L_{g*\psi_R}, \,\,\,g\in L^{2}(\RR^{2d}).
\]

When \(1\leq p'<\infty\), since \(L^2(\RR^{2d})\cap L^{p'}(\RR^{2d})\) is dense in
\(L^{p'}(\RR^{2d})\), the identity above therefore extends by density and continuity to the Banach adjoint
\[
    \PR^*:L^{p'}(\RR^{2d})\to\Scal_{p'}.
\]
It remains to identify this extension with \(g\mapsto L_{g*\psi_R}\). Let
\(g_k\in L^2(\RR^{2d})\cap L^{p'}(\RR^{2d})\) and \(g_k\to g\) in \(L^{p'}(\RR^{2d})\). Then \(g_k*\psi_R\to g*\psi_R\) in \(\Scal'(\RR^{2d})\). So we have convergence of the corresponding Schwartz kernels according to (\ref{For-kernelrhou}),
    $$k_{\mathcal{F}_{\sigma}(g_{k}*\psi_{R})}\to k_{\mathcal{F}_{\sigma}(g*\psi_{R})}\,\,\mathrm{in}\,\,\mathcal{S}^{\prime}(\RR^{2d}).
    $$
    By (\ref{For-kernelfordistribution}), it follows that for any $u,v\in \Scal(\RR^d)$,
    $$
    \langle L_{g_k*\psi_{R}}u,v\rangle
    =\langle k_{\SF(g_k*\psi_{R})},\,v\otimes u\rangle
    \to \langle k_{\SF(g*\psi_{R})},\,v\otimes u\rangle
    =\langle L_{g*\psi_{R}}u,v\rangle,
    $$
    which implies that 
    $$
    L_{g_{k}*\psi_{R}}\to L_{g*\psi_{R}}\,\,\text{in weak operator topology of }\BB(\Scal,\Scal').
    $$
    On the other hand,
\[
    L_{g_k*\psi_R}=\PR^*g_k\to \PR^*g
    \quad\text{in }\Scal_{p'}.
\]
By uniqueness of the distributional limit, \( \PR^*g=L_{g*\psi_R}\).

For \(p'= \infty\), let
\(g\in L^\infty(\RR^{2d})\). By Lemma~\ref{Lem-gLinfty}, choose
\(g_k\in\Scal(\RR^{2d})\) such that
\[
    g_k\to g \,\,\text{in }\Scal'(\RR^{2d}),
    \,\,\,
    \|g_k\|_{L^\infty}\le \|g\|_{L^\infty}.
\]
For each \(k\), the Hilbert-space computation gives
\[
\PR^* g_k=L_{g_k*\psi_R}.
\]
Since \(\PR^*\) is the Banach adjoint of \(\PR:\Scal_1\to L^1\), it is weak-\(*\) continuous from \(L^\infty=(L^1)^*\) to \(\Scal_\infty=(\Scal_1)^*\). That means 
\[
    \tr\bigl((\PR^* g_k)A^*\bigr)\to\tr\bigl((\PR^*g)A^*\bigr), \,\,\,A\in\Scal_1.
\]

At the same time, since \(g_k*\psi_R\to g*\psi_R\) in \(\Scal'(\RR^{2d})\), the distributional Weyl correspondence gives
\[
    L_{g_k*\psi_R}\to L_{g*\psi_R}
    \,\,\,
    \text{in }\mathcal L(\Scal(\RR^d),\Scal'(\RR^d)).
\]
Testing against Schwartz functions in the kernel representation, the two limits must agree. Hence
\[
    \PR^* g=L_{g*\psi_R}
    \,\,\,
    \text{for every }g\in L^\infty(\RR^{2d}).
\]

Thus, for every \(1\le p\leq\infty\), \(\PR^*\) is precisely the operator
\begin{equation}\label{Def-PRad}
    \PR^*:L^p(\RR^{2d})\to \Scal_p,\,\,\,g\mapsto L_{g*\psi_R}.
\end{equation}
Since the norm of a bounded operator equals the norm of its Banach adjoint, we obtain
\[
    \|\PR\|_{\Scal_p\to L^p}=\|\PR^*\|_{L^{p'}\to\Scal_{p'}}.
\]
Substituting the definition of \(\PR\) and combining with \eqref{Norm-PR} gives exactly \eqref{Norm-dualityPsi}.
\end{proof}

    \begin{lemma}\label{Lem-gLinfty}
	Let $g\in L^{\infty}(\RR^{2d})$. Then there exists a sequence $(g_{k})_{k}\subset\mathcal{S}(\RR^{2d})$ such that 
	$$
	g_{k}\to g\,\,\mathrm{in}\,\,\mathcal{S}'(\RR^{2d})
	\,\,\,\text{and}\,\,\, 
	\|g_{k}\|_{L^{\infty}(\RR^{2d})}\le \|g\|_{L^{\infty}(\RR^{2d})}
	\text{ for all }k.
	$$
\end{lemma}

\begin{proof}
	Let $\eta\in\mathcal{S}(\RR^{2d})$ be a standard mollifier satisfying 
	$\int_{\RR^{2d}}\eta(x)\,dx=1$. 
	For $\varepsilon>0$ define the scaled mollifier
	$$
	\eta_{\varepsilon}(x):=\varepsilon^{-2d}\eta(x/\varepsilon).
	$$
	Then for any $g\in L^{\infty}(\RR^{2d})$, the convolution 
	$g*\eta_{\varepsilon}$ is $C^{\infty}$ and satisfies
	$$
	\|g*\eta_{\varepsilon}\|_{L^{\infty}(\RR^{2d})}\le \|g\|_{L^{\infty}(\RR^{2d})},
	\,\,\,\,
	g*\eta_{\varepsilon}\to g\,\,\,\text{in }\mathcal{S}'(\RR^{2d}),
	$$
	%For any $g\in L^{\infty}(\RR^{2d})$ and any $x\in\RR^{2d}$, we have
	%\[
	%|(g*\eta_{\varepsilon})(x)|
	%=\left|\int_{\RR^{2d}} g(x-y)\eta_{\varepsilon}(y)\,dy\right|
	%\le \int_{\RR^{2d}} |g(x-y)|\,\eta_{\varepsilon}(y)\,dy.
	%\]
	%Since $|g(x-y)|\le \|g\|_{L^{\infty}}$ for all $y$,
	%\[
	%|(g*\eta_{\varepsilon})(x)|
	%\le \|g\|_{L^{\infty}}\int_{\RR^{2d}}\eta_{\varepsilon}(y)\,dy
	%=\|g\|_{L^{\infty}}.
	%\]
	%Taking the supremum over $x\in\RR^{2d}$ yields
    %	\[
	%\|g*\eta_{\varepsilon}\|_{L^{\infty}}\le \|g\|_{L^{\infty}}.
	%\]
	since for any $\phi\in\mathcal{S}(\RR^{2d})$,
	$$
	\langle g*\eta_{\varepsilon}-g,\phi\rangle
	=\langle g,\eta_{\varepsilon}*\phi-\phi\rangle\to0
	\quad\mathrm{as}\,\,\varepsilon\to0,
	$$
	because $\eta_{\varepsilon}*\phi\to\phi$ in $\mathcal{S}(\RR^{2d})$.
	
	Next, choose smooth cutoff functions 
	$\zeta_{k}\in C_{c}^{\infty}(\RR^{2d})$ such that 
	$0\le \zeta_{k}\le1$, $\zeta_{k}(x)=1$ when $|x|\le k$, 
	and $\zeta_{k}$ is supported in $\{|x|\le k+1\}$. 
	Define
	$$
	g_{k}:=(g*\eta_{\frac{1}{k}})\,\zeta_{k}.
	$$
	Then $g_{k}\in C_{c}^{\infty}(\RR^{2d})\subset\mathcal{S}(\RR^{2d})$ and 
	$$
	\|g_{k}\|_{L^{\infty}(\RR^{2d})}
	\le \|g*\eta_{\frac{1}{k}}\|_{L^{\infty}(\RR^{2d})}
	\le \|g\|_{L^{\infty}(\RR^{2d})}.
	$$
	
	It remains to show $g_{k}\to g$ in $\mathcal{S}'(\RR^{2d})$. 
	For any $\phi\in\mathcal{S}(\RR^{2d})$,
	$$
	\langle g_{k}-g,\phi\rangle
	=\langle (g*\eta_{\frac{1}{k}})\zeta_{k}-g,\phi\rangle
	=\langle g,\eta_{\frac{1}{k}}*(\zeta_{k}\phi)-\phi\rangle.
	$$
	We decompose
	$$
	\langle g,\eta_{\frac{1}{k}}*(\zeta_{k}\phi)-\phi\rangle
	=\langle g,\eta_{\frac{1}{k}}*(\zeta_{k}\phi)-\zeta_{k}\phi\rangle
	+\langle g,\zeta_{k}\phi-\phi\rangle.
	$$
	The first term tends to zero since 
	$\eta_{\frac{1}{k}}*(\zeta_{k}\phi)\to\zeta_{k}\phi$ in $\mathcal{S}(\RR^{2d})$ 
	and $g\in L^{\infty}(\RR^{2d})\subset\mathcal{S}'$.
	The second term also tends to zero because 
	$\zeta_{k}\phi\to\phi$ in $\mathcal{S}(\RR^{2d})$: indeed, for any multiindices 
	$\alpha,\beta$,
	$$
	\sup_{x\in\RR^{2d}}|x^{\alpha}\partial^{\beta}((1-\zeta_{k})\phi)(x)|
	\le \sup_{|x|\ge k}|x^{\alpha}\partial^{\beta}\phi(x)|\to0
	\quad(k\to\infty),
	$$
	by the rapid decay of $\phi$. 
	Hence $\langle g_{k}-g,\phi\rangle\to0$, which proves
	$g_{k}\to g$ in $\mathcal{S}'(\RR^{2d})$.
\end{proof}

We now deduce the compact-support estimate from the regularized estimates. This is the step where the support assumption enters. If
\(\operatorname{supp}u\subset \overline{B(0,R)}\), then for
\(v=\SF u\) one has
\begin{equation}\label{Rela-vu}
(2\pi)^{-d}v*\psi_R=\SF((\SF v)\chi_R)=\SF(u\chi_R)=\SF u=v,\,\,\,\mathrm{in}\,\,\Scal'(\RR^{2d}).
\end{equation}
Thus the regularized comparison estimates can apply directly to \(\SF u\). Consequently, for compactly supported distributions, we obtain the announced refinement of \cite[Theorem 4.2]{M26}: the loss \(R^{(5d+2)|1-2/p|}\) is replaced by \(R^{(2d+1+\varepsilon)|1-2/p|}\) for \(\varepsilon>0\).

\begin{proof}[Proof of Theorem \ref{Mainthm}]
We first record two interpolated regularized estimates. In order to do this, we recall the definitions of operators \(\PR\) and \(\PR^*\) in \eqref{Def-PR} and \eqref{Def-PRad}. Consider the \(p=2\) bounds. If \(A=L_v\in\Scal_2\), then Plancherel's theorem and Theorem~\ref{Thm-schmidthilbert} give
\[
\begin{aligned}
    \|\PR A\|_{L^2(\RR^{2d})}
    \simeq_d\|v*\psi_R\|_{L^2(\RR^{2d})}\simeq_d\|(\SF v)\chi_R\|_{L^2(\RR^{2d})}\le\|\SF v\|_{L^2(\RR^{2d})}\simeq_d\|L_v\|_{\Scal_2}=\|A\|_{\Scal_2}.
\end{aligned}
\]
Hence
\begin{equation}\label{Bound-PR2}
    \|\PR\|_{\Scal_2\to L^2}
    \lesssim_d 1 .
\end{equation}
Similarly, for \(v\in L^2(\RR^{2d})\),
\[
\begin{aligned}
    \|\PR^* v\|_{\Scal_2}=\|L_{v*\psi_R}\|_{\Scal_2}\simeq_d
    \|v*\psi_R\|_{L^2(\RR^{2d})}\lesssim_d\|\SF v\|_{L^2(\RR^{2d})}=\|v\|_{L^2(\RR^{2d})}.
\end{aligned}
\]
Thus
\begin{equation}\label{Bound-PRad2}
    \|\PR^*\|_{L^2\to\Scal_2}\lesssim_d 1 .
\end{equation}
At the endpoints, Theorem~\ref{Thm-dualityPsi} give
\begin{equation}\label{Bound-PRendpoints}
     \|\PR^*\|_{L^{q'}\to\Scal_{q'}}=\|\PR\|_{\Scal_q\to L^q(\RR^{2d})}
    \lesssim_{d,\varepsilon}
    R^{2d+1+\varepsilon},
    \,\,\,q=1,\infty,
\end{equation}

Interpolating \eqref{Bound-PRendpoints} with \eqref{Bound-PR2}, and using the
interpolation identities for \(L^p\)-spaces and Schatten classes, we obtain
\[
    \|\PR\|_{\Scal_p\to L^p}
    \lesssim_{d,p,\varepsilon}
    R^{(2d+1+\varepsilon)\left|1-\frac2p\right|},
    \qquad 1\le p\le\infty .
\]
Similarly, interpolating \eqref{Bound-PRendpoints} with \eqref{Bound-PRad2}
gives
\[
    \|\PR^*\|_{L^p\to\Scal_p}
    \lesssim_{d,p,\varepsilon}
    R^{(2d+1+\varepsilon)\left|1-\frac2p\right|},
    \qquad 1\le p\le\infty .
\]

We now translate these operator estimates back to the regularized inequalities.
By the definition of \(\PR\) in \eqref{Def-PR}, for \(A=L_v\) we have \(\PR(L_v)= (2\pi)^{-d}v*\psi_R\).
Therefore,
\begin{equation}\label{Ine-Sp1}
    \|v*\psi_R\|_{L^p(\RR^{2d})}
    \lesssim_{d,p,\varepsilon}
    R^{(2d+1+\varepsilon)\left|1-\frac2p\right|}
    \|L_v\|_{\Scal_p}.
\end{equation}
Likewise, since \(\PR^*v=L_{v*\psi_R}\) as in \eqref{Def-PRad}, we get
\begin{equation}\label{Ine-Sp2}
    \|L_{v*\psi_R}\|_{\Scal_p}
    \lesssim_{d,p,\varepsilon}
    R^{(2d+1+\varepsilon)\left|1-\frac2p\right|}
    \|v\|_{L^p(\RR^{2d})}.
\end{equation}
%We now translate these operator estimates back to the regularized inequalities. Combining with the definitions of \(\PR\) and \(\PR^*\) in \eqref{Def-PR} and \eqref{Def-PRad}, together with the above estimates gives
%\begin{equation}\label{Ine-Sp1}
    %\|v*\psi_R\|_{L^p(\RR^{2d})}
    %\lesssim_{d,p,\varepsilon}
    %R^{(2d+1+\varepsilon)\left|1-\frac2p\right|}
    %\|L_v\|_{\Scal_p},
%\end{equation}
%and 
%\begin{equation}\label{Ine-Sp2}
   % \|L_{v*\psi_R}\|_{\Scal_p}
    %\lesssim_{d,p,\varepsilon}
   % R^{(2d+1+\varepsilon)\left|1-\frac2p\right|}
    %\|v\|_{L^p(\RR^{2d})}.
%\end{equation}

We now apply these estimates to the distribution \(u\). As in the proof of \cite[Theorem~4.2]{M26}, by translation invariance we may assume without loss of generality that \(z_0=0\). Thus \(\operatorname{supp}u\subset \overline{B(0,R)}\). Then in \(\Scal'(\RR^{2d})\), we have \eqref{Rela-vu}.

We first prove \eqref{Maininequality1}.  If \(\|L_{\SF u}\|_{\Scal_p}=\infty\), there is nothing to prove. Otherwise, applying \eqref{Ine-Sp1} to \(v=\SF u\) and using \eqref{Rela-vu}, we obtain
\[
\|\SF u\|_{L^p(\RR^{2d})}
=\|v\|_{L^p(\RR^{2d})}
\simeq_d\|v*\psi_R\|_{L^p(\RR^{2d})}
\lesssim_{d,p,\varepsilon} R^{(2d+1+\varepsilon)\left|1-\frac{2}{p}\right|}\|L_{v}\|_{\Scal_p}
= R^{(2d+1+\varepsilon)\left|1-\frac{2}{p}\right|}\|L_{\SF u}\|_{\Scal_p}.
\]

For \eqref{Maininequality2}, if \(\|\SF u\|_{L^p(\RR^{2d})}=\infty\), there is nothing to prove. Otherwise \(v=\SF u\in L^p(\RR^{2d})\), and applying \eqref{Ine-Sp2} to the same \(v\), again using \eqref{Rela-vu}, gives
\[
\|L_{\SF u}\|_{\Scal_p}
=\|L_v\|_{\Scal_p}
\simeq_d\|L_{v*\psi_R}\|_{\Scal_p}
\lesssim_{d,p,\varepsilon} R^{(2d+1+\varepsilon)\left|1-\frac{2}{p}\right|}\|v\|_{L^p(\RR^{2d})}
=R^{(2d+1+\varepsilon)\left|1-\frac{2}{p}\right|}\|\SF u\|_{L^p(\RR^{2d})}.
\]
The proof is complete.
\end{proof}

%This completes the proof of the refined restriction estimate. The polynomial dependence on \(R\) obtained above comes from the cutoff estimates for \(\chi_R\) and the interpolation between the endpoint bounds. 

We next introduce auxiliary constants in order to relate this dependence to lower bounds in both directions of the estimate. Recall the definitions of \(C_{d,p}(R)\) and \(D_{d,p}(R)\) in \eqref{Def-CpR} and \eqref{Def-DpR}, we introduce the regularized
comparison constants
\[
\begin{aligned}
    \widetilde{C}_p(R)&
%\defeq\|g\mapsto L_{g*\psi_R}\|_{L^p(\RR^{2d})\to \Scal_p}
\defeq\sup\left\{\frac{\|L_{g*\psi_R}\|_{\Scal_p}}{\|g\|_{L^p}}: g\in L^p(\RR^{2d}),\,g\neq0\right\}\\
\widetilde D_{d,p}(R)&\defeq\sup\left\{\frac{\|v*\psi_R\|_{L^p}}{\|L_v\|_{\Scal_p}}
:v\in\Scal'(\RR^{2d}),\,0<\|L_v\|_{\Scal_p}<\infty\right\}.
\end{aligned}
\]

These two constants will only serve as regularized intermediaries. They allow us to compare the compact-support constants with regularized operator norms, where Schatten duality gives a direct relation between the two directions.

\begin{proposition}\label{Prop-comparison-constants}
Let \(1\le p\le\infty\). Then, for every \(R\ge1\),
\begin{equation}\label{Rela-CCtilde}
C_{d,p}(R)\lesssim_d \widetilde C_{d,p}(2R) \lesssim_{\chi,d} C_{d,p}(2R).
\end{equation}
Similarly,
\begin{equation}\label{Rela-CCtildeprime}
D_{d,p}(R)\lesssim_d \widetilde D_{d,p}(2R)\lesssim_{\chi,d} D_{d,p}(2R).
\end{equation}
Moreover, 
\begin{equation}\label{Relat-CCprime} 
\widetilde  C_{d,p}(R)\simeq_d \widetilde D_{p'}(R). 
\end{equation}
\end{proposition}

\begin{proof}
We prove \eqref{Rela-CCtilde}; the proof of
\eqref{Rela-CCtildeprime} is analogous.

Let \(u\in\mathcal E'(\RR^{2d})\) satisfy \(\operatorname{supp}u\subset \overline{B(0,2R)}\). Set \(g\defeq \SF u\). Since \(u\chi_{2R}=u\), by \eqref{Conv-sym} we have
\[
    (2\pi)^{-d}g*\psi_{2R}
    =\SF\big((\SF g)\chi_{2R}\big)
    =\SF(u\chi_{2R})
    = \SF u
    = g .
\]
Hence
\[
    \frac{\|L_{\SF u}\|_{\Scal_p}}{\|\SF u\|_{L^p}} = (2\pi)^{-d}\frac{\|L_{g*\psi_{2R}}\|_{\Scal_p}}{\|g\|_{L^p}} \le\widetilde C_{d,p}(2R).
\]
Taking the supremum over all such \(u\), we obtain
\[
C_{d,p}(R)\lesssim_d \widetilde C_{d,p}(2R).
\]

Conversely, let \(g\in L^p(\RR^{2d})\), and define \(u\defeq \SF(g*\psi_R)=(2\pi)^d(\SF g)\chi_R\).
In particular,
\[
\operatorname{supp}u\subset \operatorname{supp}\chi_R\subset \overline{B(0,2R)},\,\,\,
\SF u=g*\psi_R.
\]
Thus, by the definition of \( C_{d,p}(R)\) and Young's inequality,
\[
\|L_{g*\psi_R}\|_{\Scal_p}
=\|L_{\SF u}\|_{\Scal_p}
\le C_{d,p}(R)\|\SF u\|_{L^p}
=C_{d,p}(R)\|g*\psi_R\|_{L^p}\le C_{d,p}(R)\|g\|_{L^p}\|\psi_R\|_{L^1}.
\]
The scaling of \(\SF\) gives
\[
\|\psi_R\|_{L^1(\RR^{2d})}=\|\SF\chi_R\|_{L^1(\RR^{2d})}=\|\SF\chi\|_{L^1(\RR^{2d})}=(2\pi)^{-d}\|\widehat\chi\|_{L^1(\RR^{2d})}.
\]
Therefore
\[
\|L_{g*\psi_R}\|_{\Scal_p}
\lesssim_{\chi,d}C_{d,p}(R)\|g\|_{L^p}.
\]
Taking the supremum over all \(g\in L^p(\RR^{2d})\), and replacing \(R\) by \(2R\) gives the second inequality in
\eqref{Rela-CCtilde}.

The identity \eqref{Relat-CCprime} follows from the same duality argument as in Theorem~\ref{Thm-dualityPsi}.
\end{proof}

For \(1\le p\le\infty\), recall the definitions of \(\gamma_{d,p}\) and \(\eta_{d,p}\) in \eqref{Def-gammaeta}, and define the other two polynomial growth exponents as follows:
\[
\widetilde\gamma_{d,p}\defeq\inf\left\{s\ge0:\sup_{R\ge1}\frac{\widetilde C_{d,p}(R)}{R^s}<+\infty\right\},\,\,\,
\widetilde\eta_{d,p}\defeq\inf\left\{s\ge0:\sup_{R\ge1}\frac{\widetilde D_{d,p}(R)}{R^s}<+\infty\right\}.
\]

Combining with Theorem \ref{Thm-oneway}, Theorem \ref{Thm-dualityPsi} and Proposition~\ref{Prop-comparison-constants}, we obtain the corresponding identity for the polynomial growth exponents.
\begin{corollary}\label{Cor-gammap}
For every \(1\le p\le\infty\) and any \(\varepsilon>0\),
\[
\gamma_{d,p}=\widetilde\gamma_{d,p}=\widetilde\eta_{d,p^{\prime}}=\eta_{d,p'}\leq (2d+1+\varepsilon)\left|1-\frac2p\right|.
\]
\end{corollary}

The preceding comparison allows us to transfer lower bounds between the two directions of the estimate. In the next section we construct explicit compactly supported examples and use them, together with Corollary~\ref{Cor-gammap}, to show that the dependence on the support radius is nontrivial in certain ranges of \(p\).

\section{Lower bounds for the dependence on the support radius}\label{Section-example}\label{Section-lowerbound}
 In this section, we construct compactly supported phase-space distributions for which the ratio between the \(L^p\)-norm of the symplectic Fourier transform and the corresponding Schatten norm grows like a positive power of \(R\) in certain ranges of \(p\), which implies the polynomial dependence on \(R\) in the previous section is unavoidable.

The examples are obtained by truncating Laguerre functions. Since \(\rho(\varphi_k)=P_k\), the Schatten norm is explicit. Moreover, the decay estimates in Lemma~\ref{Lem-Derivativephik} imply that truncation beyond the natural scale \(\sqrt{k}\) changes the relevant norms only by an exponentially small error. We begin by making this cutoff estimate precise.

\begin{lemma}\label{Lem-chivarphi}
Let \(N\in\NN\). There exists \(c>0\) such that for all \(k\) sufficiently large so that
\[
16k\ge 6(2k+d+N),
\]
and all \(R\ge 4\sqrt{k}\), we have the following estimate:
\[
\|(1-\chi_R)\varphi_k\|_{(N,2)}\lesssim_{d,N,\chi}\ee^{-cR^2}.
\]
\end{lemma}

\begin{proof}
Fix multi-indices \(\alpha,\beta\in\NN^{2d}\) with \(|\alpha|+|\beta|\le N\), 
we need to estimate
\[
\left\|\partial^\alpha\bigl(z^\beta(1-\chi_R)\varphi_k\bigr)\right\|_{L^2(\RR^{2d})}.
\]
By Leibniz's rule,
\[
\partial^\alpha\bigl(z^\beta(1-\chi_R)\varphi_k\bigr)=\sum_{\eta_1+\eta_2+\eta_3=\alpha}C_{\alpha,\eta_1,\eta_2,\eta_3}\,\partial^{\eta_1}(1-\chi_R)\,\partial^{\eta_2}(z^\beta)\,\partial^{\eta_3}\varphi_k .
\]
For \(\eta_1\neq0\), the scaling \(\chi_R(z)=\chi(z/R)\) gives
\[
\partial^{\eta_1}(1-\chi_R)(z)=-R^{-|\eta_1|}(\partial^{\eta_1}\chi)(z/R).
\]
Hence \(\partial^{\eta_1}(1-\chi_R)\) is supported in the annulus
\[
\{z\in\RR^{2d}:R\le |z|\le 2R\},
\]
and satisfies
\[
|\partial^{\eta_1}(1-\chi_R)(z)|\lesssim_{\eta_1,\chi}R^{-|\eta_1|}.
\]
For \(\eta_1=0\), the factor \(1-\chi_R\) is supported in \(\{z\in\RR^{2d}:|z|\ge R\}\).
Thus every term in the above Leibniz expansion is supported in the region \(\{z\in\RR^{2d}:|z|\ge R\}\).

Since \(|\eta_3|\le |\alpha|\le N\), if \(R\ge4\sqrt{k}\), then \(R^2\ge16k\ge 6(2k+d+N)\ge 6(2k+d+|\eta_3|)\).
Consequently, on the support of each summand,
\[
|z|\ge R\ge \sqrt{6(2k+d+|\eta_3|)}.
\]
Lemma~\ref{Lem-Derivativephik} therefore gives
\[
|\partial^{\eta_3}\varphi_k(z)|\lesssim_{d,N}\ee^{-\frac{\gamma}{4}|z|^2},\,\,\,|z|\ge R.
\]

Moreover, since \(|\beta|\le N\),
\[
|\partial^{\eta_2}(z^\beta)|\lesssim_{\beta,N}(1+|z|)^N.
\]
Combining the preceding estimates, each summand in the Leibniz expansion is bounded by
\[
C_{d,N,\chi}(1+|z|)^N\ee^{-\frac{\gamma}{4}|z|^2}\mathbf 1_{\{|z|\ge R\}}.
\]
Absorbing the polynomial factor into the Gaussian decay, we obtain, 
\[
\left|\partial^\alpha\bigl(z^\beta(1-\chi_R)\varphi_k\bigr)(z)\right|
\lesssim_{d,N,\chi}\ee^{-\frac\gamma 8|z|^2}\mathbf 1_{\{|z|\ge R\}}.
\]
Therefore
\[
\begin{aligned}
\left\|
\partial^\alpha\bigl(z^\beta(1-\chi_R)\varphi_k\bigr)
\right\|_{L^2(\RR^{2d})}
\lesssim_{d,N,\chi}
\left(\int_{|z|\ge R}\ee^{-\frac\gamma 4|z|^2}\,\dd z\right)^{1/2} \lesssim_{d,N,\chi}
\ee^{-cR^2},
\end{aligned}
\]
for some \(c>0\). Finally, summing over all \(\alpha,\beta\in\NN^{2d}\) with \(|\alpha|+|\beta|\le N\), we obtain
\[
\|(1-\chi_R)\varphi_k\|_{(N,2)}
=\sum_{|\alpha|+|\beta|\le N}\left\|\partial^\alpha\bigl(z^\beta(1-\chi_R)\varphi_k\bigr)\right\|_{L^2(\RR^{2d})}
\lesssim_{d,N,\chi}\ee^{-cR^2},
\]
as desired.
\end{proof}

%We shall also need to control the \(L^1\)-norm of the symplectic Fourier transform of the cutoff error. For this purpose we use the standard fact that the \(L^1\)-norm of a Fourier transform is controlled by a sufficiently high Sobolev norm. We record the statement for completeness and omit the proof.

 %\begin{lemma}\label{Lem-SobolevFourier}
 	%For $g\in\mathcal{S}(\RR^{2d})$, we have 
 	%$$
 	%\|\widehat{g}\|_{L^{1}(\RR^{2d})}\lesssim_{d}\|g\|_{H^{d+2}(\RR^{2d})}.%\leq\|g\|_{(2d+2,2)}.
 	%$$
 %\end{lemma}
\iffalse
==============================fouriersobolevnorm==================================
\begin{proof}
	We write
	$$
	|\widehat{g}(w)|
	=(1+|w|^{2})^{-s/2}\,(1+|w|^{2})^{s/2}|\widehat{g}(w)|.
    $$
	%Hence
	%$$
	%\|\widehat{g}\|_{L^{1}(\RR^{2d})}
	%=\int_{\RR^{2d}}(1+|w|^{2})^{-s/2}\,
	%(1+|w|^{2})^{s/2}|\widehat{g}(w)|\,\mathrm{d}w.
	%$$
	Applying H\"older's inequality gives
	$$
\|\widehat{g}\|_{L^{1}(\RR^{2d})}
	\le 
	\|(1+|w|^{2})^{-s/2}\|_{L^{2}(\RR^{2d})}
	\,
	\|(1+|w|^{2})^{s/2}\widehat{g}\|_{L^{2}(\RR^{2d})}.
	$$
	The second factor is exactly the Sobolev norm
	\(\|g\|_{H^{s}(\RR^{2d})}\).
	Since $s>d$, we have
	$$
	\|(1+|w|^{2})^{-s/2}\|_{L^{2}(\RR^{2d})}<\infty.
	$$
	Therefore,
$$
		\|\widehat{g}\|_{L^{1}(\RR^{2d})}
		\lesssim_{d,s}\,\|g\|_{H^{s}(\RR^{2d})},
		\,\,\, s>d.
$$
Taking $s=d+2$, then we have the desired estimate.
%$$
%\|\widehat{g}\|_{L^{1}(\RR^{2d})}\lesssim_{d}\,\|g\|_{H^{d+2}(\RR^{2d})},%\leq\|g\|_{(2d+2,2)},
%$$
%as desired.
\end{proof}
==============================================================================
\fi
We now build the example and give the following proposition. 
%The uncut Laguerre function \(\widetilde u=\varphi_{k_0}\) is not compactly supported, but it has two crucial advantages: its Weyl operator is exactly the Hermite projection \(P_{k_0}\), and the \(L^p\)-norms of \(\SF\widetilde u\) are explicitly known from Lemmas~\ref{Lem-Lpnormphik} and \ref{Lem-Laguerrefourier}. Multiplying by \(\chi_R\), with \(R=4\sqrt{k_0}\), gives a compactly supported example without changing the leading-order behavior when \(R\) is large according to Lemma \ref{Lem-chivarphi}.

\begin{proposition}\label{Prop-explicitexample}
Let \(p_c\defeq \frac{4d}{2d-1}\) and 
\[
\widetilde u(z)\defeq \varphi_{k_0}(z),
\,\,\,
u(z)\defeq \chi_R(z)\widetilde u(z),
\,\,\,
R=4\sqrt{k_0},
\]
where \(k_0\) is sufficiently large. 
Then, for \(1\le p\le\infty\) such that \(p\neq p_c\), with the convention \(1/\infty=0\),
\begin{equation}\label{Ratio-R}
\frac{\|\SF u\|_{L^p(\RR^{2d})}}{\|L_{\SF u}\|_{\Scal_p}}
\simeq_{d,p}
\begin{cases}
R^{\,2(d-1)-\frac{2(2d-1)}{p}}, & p_c<p\le\infty,\\[4pt]
R^{\,\frac2p-1}, & 1\le p<p_c.
\end{cases}
\end{equation}
\end{proposition}

\begin{proof}
We first estimate the Schatten norm of \(L_{\SF \widetilde u}\) and \(L^p\)-norm of \(\SF \widetilde u\). By the Hermite--Laguerre
correspondence \eqref{Ide-projection},
\[
L_{\SF\widetilde u}
=\rho(\SF(\SF\widetilde u))
=\rho(\widetilde u)=\rho(\varphi_{k_0})=P_{k_0},
\]
Consequently,
\begin{equation}\label{Est-Spnormexplictexample}
\|L_{\SF\widetilde u}\|_{\Scal_p}
=\|P_{k_0}\|_{\Scal_p}
=d_{k_0}^{1/p}
\simeq_d k_0^{\frac{d-1}{p}},
\end{equation}
with the convention that \(d_{k_0}^{1/\infty}=1\) and using \eqref{For-normpPk}.

By Lemma~\ref{Lem-Laguerrefourier},
\[
\SF\widetilde u(z)
=\SF\varphi_{k_0}(z)
=2^d(-1)^{k_0}\varphi_{k_0}(2z).
\]
Therefore
\begin{equation}\label{Est-Lpnormexplictexample}
\|\SF\widetilde u\|_{L^p(\RR^{2d})}
=2^d\|\varphi_{k_0}(2\cdot)\|_{L^p(\RR^{2d})}
=2^{d-\frac{2d}{p}}\|\varphi_{k_0}\|_{L^p(\RR^{2d})}
\simeq_{d,p}\|\varphi_{k_0}\|_{L^p(\RR^{2d})},
\end{equation}
which are also polynomial terms according to Lemma~\ref{Lem-Lpnormphik}.

Next we show that replacing \(\widetilde u\) by \(u=\chi_R\widetilde u\) produces only an exponentially small error, both on the Schatten side and on the \(L^p\)-side. Lemma~\ref{Lem-chivarphi}, with \(N=d+2\), gives
\begin{equation}\label{Est-N2normuutilde}
\|u-\widetilde u\|_{(d+2,2)}=\|(1-\chi_R)\varphi_{k_0}\|_{(d+2,2)}\lesssim_{d,\chi}\ee^{-cR^2}.
\end{equation}
Combining with Lemma~\ref{Lem-S1normN2norm} and the continuous inclusion relation \eqref{Rela-Inclusion}, for any \(1\leq p\leq \infty\),
\begin{equation}\label{error}
\|L_{\SF u}-L_{\SF\widetilde u}\|_{\Scal_p}\leq\|L_{\SF u}-L_{\SF\widetilde u}\|_{\Scal_1}
\lesssim\|u-\widetilde u\|_{(d+2,2)}
\lesssim\ee^{-cR^2}=\ee^{-16ck_0}.
\end{equation}
%On the other hand, by the same argument as that in Lemma \ref{Lem-chivarphi},
%\[
%\|u-\widetilde u\|_{L^1(\RR^{2d})}\lesssim\|u-\widetilde u\|_{(d+2,2)}\lesssim_{d,\chi}\ee^{-cR^2}.
%\]
%Moreover, by the \(L^1\to\mathcal B(L^2)\) bound for the integratedSchr\"odinger representation (see e.g. \cite[proposition 3.1]{LS25}),
%\[\|L_{\SF u}-L_{\SF\widetilde u}\|_{\Scal_\infty}=\|\rho(u-\widetilde u)\|_{\mathcal B(L^2)}\lesssim\|u-\widetilde u\|_{L^1(\RR^{2d})}\lesssim\ee^{-cR^2}.\]Interpolating between \(\Scal_1\) and \(\Scal_\infty\), we get, for\(1\le p\le\infty\),\[\|L_{\SF u}-L_{\SF\widetilde u}\|_{\Scal_p}\lesssim\ee^{-cR^2}=\ee^{-16ck_0}.\]

On the other hand,
\[
\|\SF(u-\widetilde u)\|_{L^\infty(\RR^{2d})}
\le\|u-\widetilde u\|_{L^1(\RR^{2d})}
\lesssim\ee^{-cR^2}.
\]
Moreover, using the standard Sobolev estimate
\[
    \|\widehat g\|_{L^1(\RR^{2d})}\lesssim_d\|g\|_{H^{d+2}(\RR^{2d})},\,\,\, g\in\Scal(\RR^{2d}),
\]
together with \eqref{Rela-twotransforms}, \eqref{Def-N2norm} and \eqref{Est-N2normuutilde},
\[
\|\SF(u-\widetilde u)\|_{L^1(\RR^{2d})}\simeq_d\|\widehat{u-\widetilde u}\|_{L^1(\RR^{2d})}
\lesssim_d\|u-\widetilde u\|_{H^{d+2}(\RR^{2d})}
\lesssim\|u-\widetilde u\|_{(d+2,2)}
\lesssim\ee^{-cR^2}.
\]
Interpolating between \(L^1\) and \(L^\infty\), we obtain
\[
\|\SF(u-\widetilde u)\|_{L^p(\RR^{2d})}
\lesssim\ee^{-cR^2}=\ee^{-16ck_0},
\,\,\, 1\le p\le\infty.
\]
This error and the one in \eqref{error}, are exponentially small and therefore negligible compared with the polynomial main terms. Therefore, together with \eqref{Est-Spnormexplictexample} and \eqref{Est-Lpnormexplictexample},
\[
\|L_{\SF u}\|_{\Scal_p}
\simeq_d\|L_{\SF\widetilde u}\|_{\Scal_p}
\simeq_d k_0^{\frac{d-1}{p}}.
\]
and 
\[
\|\SF u\|_{L^p(\RR^{2d})}\simeq_d\|\SF\widetilde u\|_{L^p(\RR^{2d})}\simeq_{d,p}\|\varphi_{k_0}\|_{L^p(\RR^{2d})}.
\]

By Lemma~\ref{Lem-Lpnormphik}, together with the two estimates above, we obtain the explicit estimate 
\begin{equation*}%\label{Ratio-k}
\frac{\|\SF u\|_{L^p(\RR^{2d})}}{\|L_{\SF u}\|_{\Scal_p}}
\simeq_{d,p}
\begin{cases}
k_0^{\,d-1-\frac{2d-1}{p}}, & p_c<p\le\infty,\\[4pt]
k_0^{\,\frac1p-\frac12}, & 1\le p<p_c.
\end{cases}
\end{equation*}
Substituting $R=4\sqrt{k_{0}}$ yields (\ref{Ratio-R}).
\end{proof}

Combining the explicit lower bounds from Proposition~\ref{Prop-explicitexample} with the comparison of growth exponents in Corollary~\ref{Cor-gammap} gives the promised nontriviality statement for the dependence on \(R\).

\begin{proof}[Proof of Theorem \ref{Thm-nontrivial} and Corollary \ref{Cor-nontrivial}]
Firstly, by the definitions of \(\gamma_{d,p}\) and \(\eta_{d,p}\) in \eqref{Def-gammaeta}, Theorem \ref{Mainthm} gives the upper bounds.

By Proposition~\ref{Prop-explicitexample} and definition of \(a_{d,p}\) in \eqref{Def-ap}, we have 
\[
\eta_{d,p}\ge \max\{a_{d,p},0\}, \,\,\, \gamma_{d,p}\ge \max\{-a_{d,p},0\}. 
\]
Combining this with Corollary~\ref{Cor-gammap}, namely \(\gamma_{d,p}=\eta_{d,p'}\), gives \begin{equation*}
\gamma_{d,p}\ge \max\{a_{d,p'},-a_{d,p},0\},\,\,\,\eta_{d,p}\ge \max\{a_{d,p},-a_{d,p'},0\},
\end{equation*}
while the duality contribution for \(\eta_{d,p}\) does not give any additional information. Thus, we obtain \eqref{For-etagamma}.

A direct check from the definition of \(a_{d,p}\) shows that \(\underline{\gamma_{d,p}}>0\) for \eqref{Range1} and that \(\underline{\eta_{d,p}}>0\) for \eqref{Range2}. This proves the two nontriviality statements. Indeed, if the corresponding estimate held with a constant independent of \(R\), then the associated constant \(C_{d,p}(R)\) or \(D_{d,p}(R)\) would be uniformly bounded in \(R\), forcing \(\underline{\gamma_{d,p}}=0\) or \(\underline{\eta_{d,p}}=0\), respectively.
\end{proof}

\begin{remark}
Theorem~\ref{Thm-nontrivial} also shows that the upper bound for the reverse
comparison estimate becomes asymptotically sharp in high dimensions. Indeed,
fix \(p>2\). Then, for all sufficiently large \(d\), the lower bound for
\(\eta_{d,p}\) is given by
\[
    \underline{\eta_{d,p}}=a_{d,p}=2(d-1)-\frac{2(2d-1)}{p}.
\]
Since
\[
    \underline{\eta_{d,p}}
    \leq \eta_{d,p}
    \leq U_{d,p},
\]
from Theorem \ref{Thm-nontrivial}, we have
\[
    1
    \leq \frac{U_{d,p}}{\eta_{d,p}}
    \leq
    \frac{U_{d,p}}{\underline{\eta_{d,p}}}
    \to 1,\,\,\,
    \text{as } d\to\infty.
\]
Thus, for every fixed \(p>2\), the exponent in the reverse comparison estimate
\eqref{Maininequality1} is asymptotically optimal as \(d\to\infty\).
%Proposition~\ref{Prop-explicitexample} also shows that the reverse comparison estimate \eqref{Maininequality1} is close to optimal for \(p>2+\frac1{d-1}\). %Indeed, in this range, \(a_p>0\), and the gap between the upper exponent in \eqref{Maininequality1} and the lower-bound exponent is
%\[
    %(2d+1+\varepsilon)\left(1-\frac2p\right)-2(d-1)+\frac2p(2d-1)=3-\frac4p+\varepsilon\left(1-\frac2p\right),\,\,\,\text{for any }\varepsilon>0.
%\]
%Equivalently, the upper bound differs from the lower bound by at most the additional factor \(R^{4-6/p}\). This factor is bounded by \(R^4\), and is strictly smaller than \(R^4\) for every finite \(p\) in this range.
\end{remark}

\section*{Acknowledgements}
The author is deeply grateful to Alessio Martini for his valuable suggestions and many insightful discussions during the preparation of this work. The author also gratefully acknowledges the financial support of the China Scholarship Council (Grant No. 202406290140).

\vspace{0.5em}

\begingroup
\footnotesize
\setlength{\parindent}{0pt}
\setlength{\parskip}{3pt}
\textsc{Jie Liu}\\
School of Mathematics and Statistics, Northwestern Polytechnical University, Xi'an 710129, China\\
Dipartimento di Scienze Matematiche, Politecnico di Torino, Corso Duca degli Abruzzi 24, 10129 Torino, Italy\\
\emph{E-mail address}: \texttt{jie.liu@polito.it}

\endgroup

\begin{liucomment}
    ================ Standard Euclidean Fourier transform of a Gaussian==============
    Write $z=(z_{1},\dots,z_{2d})$, $w=(w_{1},\dots,w_{2d})$. Then
	$$
	\mathrm{e}^{-a|z|^{2}} \mathrm{e}^{-iw\cdot z}
	= \prod_{j=1}^{2d} \mathrm{e}^{-a z_{j}^{2}}\mathrm{e}^{-i w_{j} z_{j}},
	$$
	and the integral factors:
	$$
	\int_{\RR^{2d}}\mathrm{e}^{-a|z|^{2}}\mathrm{e}^{-iw\cdot z}\,\mathrm{d}z
	= \prod_{j=1}^{2d}
	\left( \int_{\mathbb R} \mathrm{e}^{-a x^{2}} \mathrm{e}^{-i w_{j} x}\,\mathrm{d}x \right).
	$$
	Claim that
	$$
	\int_{\mathbb R} \mathrm{e}^{-a x^{2}} \mathrm{e}^{-i w x}\,\mathrm{d}x
	= \sqrt{\frac{\pi}{a}}\,\mathrm{e}^{-\frac{w^{2}}{4a}}.
	$$
	Then the desired identity can be proven.
	
	From now on, we prove the claim under the standard hypothesis $\Re(a)>0,$
	which guarantees absolute convergence of the integral and decay of the Gaussian at infinity. Write the quadratic form in the exponent as
	$$
	-a x^{2} - i w x = -a\Big(x + \frac{i w}{2a}\Big)^{2} - \frac{w^{2}}{4a}.
	$$
	Thus the integrand factorizes:
	$$
	\mathrm{e}^{-a x^{2}} e^{-i w x}
	=\mathrm{e}^{-a\left(x + \frac{i w}{2a}\right)^{2}}\, \mathrm{e}^{-\frac{w^{2}}{4a}}.
	$$
	Then by change of variable
	$$
	u = x + \frac{i w}{2a}, \qquad x = u - \frac{i w}{2a},\qquad \mathrm{d}x = \mathrm{d}u,
	$$
	the real line $x\in\mathbb{R}$ is carried to the horizontal line $\{u\in\mathbb{C}:\; \Im u = \tfrac{w}{2a\,'}\}$ in the complex $u$-plane. Because $\Re(a)>0$, the function $u\mapsto \mathrm{e}^{-a u^{2}}$ is entire and decays exponentially
	as $|\Re u|\to\infty$ along any horizontal line in the complex plane. Consequently,
	by Cauchy's theorem, we may shift the contour of integration back to the real axis without changing the integral:
	$$
	\int_{x\in\mathbb{R}} \mathrm{e}^{-a\left(x + \frac{i w}{2a}\right)^{2}}\,\mathrm{d}x
	= \int_{u\in\mathbb{R}} \mathrm{e}^{-a u^{2}}\,\mathrm{d}u.
	$$
	Therefore
	$$
	\int_{\mathbb{R}} \mathrm{e}^{-a x^{2}} \mathrm{e}^{-i w x}\,\mathrm{d}x
	= \mathrm{e}^{-\frac{w^{2}}{4a}} \int_{\mathbb{R}} \mathrm{e}^{-a u^{2}}\,\mathrm{d}u= \sqrt{\frac{\pi}{a}}\, \mathrm{e}^{-\frac{w^{2}}{4a}},
	$$
	which proves the claim.
    ====================================================================
\end{liucomment}

		\end{document}